\documentclass[11pt,onecolumn,letterpaper,twoside]{IEEEtran}

\usepackage{graphicx}				
\usepackage{amssymb,amsmath,amsthm}

\usepackage{algorithm}

\newtheorem{Def}{Definition}

\newtheorem{Assumption}{Assumption}

\newtheorem{Thm}{Theorem}

\newtheorem{Lem}{Lemma}
\newtheorem{Cor}{Corollary}
\newtheorem{Rem}{Remark}

\DeclareMathOperator*{\argmin}{argmin}

\title{A Primal-Dual Type Algorithm with the $O(1/t)$ Convergence Rate  for Large Scale Constrained Convex Programs}

\author{Hao Yu and Michael J. Neely \\Department of Electrical Engineering \\
University of Southern California}

\begin{document}
\maketitle

\begin{abstract}
This paper considers large scale constrained convex programs, which are usually not solvable by interior point methods or other Newton-type methods due to the prohibitive computation and storage complexity for Hessians and matrix inversions.  Instead, large scale constrained convex programs are often solved by gradient based methods or decomposition based methods.  The conventional primal-dual subgradient method, aka, Arrow-Hurwicz-Uzawa subgradient method, is a low complexity algorithm with the $O(1/\sqrt{t})$ convergence rate, where $t$ is the number of iterations. If the objective and constraint functions are separable, the Lagrangian dual type method can decompose a large scale convex program into multiple parallel small scale convex programs. The classical dual gradient algorithm is an example of Lagrangian dual type methods and has convergence rate $O(1/\sqrt{t})$. Recently, a new Lagrangian dual type algorithm with faster $O(1/t)$ convergence is proposed in \cite{YuNeely15ArXivGeneralConvex}.  However, if the objective or constraint functions are not separable, each iteration of the Lagrangian dual type method in \cite{YuNeely15ArXivGeneralConvex} requires to solve a large scale unconstrained convex program, which can have huge complexity. This paper proposes a new primal-dual type algorithm, which only involves simple gradient updates at each iteration and has the $O(1/t)$ convergence rate.   

\end{abstract}


\section{Introduction}\label{sec:intro}
Fix positive integers $n$ and $m$, which are typically large. 
Consider the general constrained convex program:
\begin{align}
\text{minimize:} \quad &f(\mathbf{x}) \label{eq:program-objective}\\
\text{such that:} \quad  &  g_k(\mathbf{x}) \leq  0, \forall k\in\{1,2,\ldots,m\} \label{eq:program-inequality-constraint}\\
			 &  \mathbf{x}\in \mathcal{X} \label{eq:program-set-constraint}
\end{align}
where set $\mathcal{X}\subseteq \mathbb{R}^{n}$ is a compact convex set;  function $f(\mathbf{x})$ is convex and smooth on $\mathcal{X}$; and functions $g_k(\mathbf{x}),\forall k \in\{1,2,\ldots,m\}$ are convex, smooth and Lipschitz continuous on $\mathcal{X}$. Denote the stacked vector of multiple functions $g_1(\mathbf{x}), g_2(\mathbf{x}), \ldots, g_m(\mathbf{x})$ as $\mathbf{g}(\mathbf{x}) = \big[g_1(\mathbf{x}), g_2(\mathbf{x}), \ldots, g_m(\mathbf{x})\big]^T$.  The Lipschitz continuity of each $g_{k}(\mathbf{x})$ implies that $\mathbf{g}(\mathbf{x})$ is Lipschitz continuous on $\mathcal{X}$. Throughout this paper, we have the following assumptions on convex program  \eqref{eq:program-objective}-\eqref{eq:program-set-constraint}:

\begin{Assumption}[Basic Assumptions] \label{as:basic}~
\begin{itemize}
\item  There exists a (possibly non-unique) optimal solution $\mathbf{x}^\ast\in \mathcal{X}$ that solves convex program \eqref{eq:program-objective}-\eqref{eq:program-set-constraint}. 
\item  There exists $L_f\geq0$ such that $\Vert \nabla f(\mathbf{x}) - \nabla f(\mathbf{y})\Vert \leq L_{f}\Vert \mathbf{x} - \mathbf{y}\Vert$ for all $\mathbf{x}, \mathbf{y}\in \mathcal{X}$, i.e., $f(\mathbf{x})$ is smooth with modulus $L_f$. For each $k\in\{1,2,\ldots,m\}$, there exists $L_{g_k}\geq0$ such that $\Vert \nabla g_{k}(\mathbf{x}) - \nabla g_{k}(\mathbf{y})\Vert \leq  L_{g_{k}}\Vert \mathbf{x} - \mathbf{y}\Vert$ for all $\mathbf{x}, \mathbf{y}\in \mathcal{X}$, i.e., $g_{k}(\mathbf{x})$ is smooth with modulus $L_{g_{k}}$. Denote $\mathbf{L}_{\mathbf{g}} = [L_{g_{1}}, \ldots, L_{g_{m}}]^{T}$. 
\item There exists a constant $\beta$ such that $\Vert \mathbf{g}(\mathbf{x}) - \mathbf{g}(\mathbf{y})\Vert \leq \beta \Vert \mathbf{x} - \mathbf{y}\Vert$ for all $\mathbf{x}, \mathbf{y} \in \mathcal{X}$, i.e., $\mathbf{g}(\mathbf{x})$ is Lipschitz continuous with modulus $\beta$. 
\item There exists a constant $C$ such that $\Vert \mathbf{g}(\mathbf{x})\Vert \leq C$ for all $\mathbf{x}\in \mathcal{X}$.  
\item There exists a constant $R$ such that $\Vert \mathbf{x} - \mathbf{y}\Vert\leq R$ for all $\mathbf{x}, \mathbf{y}\in \mathcal{X}$.
\end{itemize}
\end{Assumption}
Note that the existence of $C$ follows directly from the continuity of $\mathbf{g}(\mathbf{x})$ and the compactness of set $\mathcal{X}$. The existence of $R$ follows directly from the compactness of set $\mathcal{X}$.

\begin{Assumption}[Existence of Lagrange multipliers] \label{as:strong-duality}
 There exists a Lagrange multiplier vector $\boldsymbol{\lambda}^\ast = [\lambda_1^\ast, \lambda_2^\ast, \ldots, \lambda_m^\ast]\geq \mathbf{0}$ attaining the strong duality for problem \eqref{eq:program-objective}-\eqref{eq:program-set-constraint}, i.e., 
\begin{align*}
q(\boldsymbol{\lambda}^\ast) = \min\limits_{\mathbf{x}\in \mathcal{X}}\left\{f(\mathbf{x}) : g_k(\mathbf{x})\leq 0, \forall k\in\{1,2,\ldots,m\}\right\}, 
\end{align*}
where $q(\boldsymbol{\lambda}) = \min\limits_{\mathbf{x}\in \mathcal{X}}\{f(\mathbf{x})+ \sum_{k=1}^m \lambda_k g_k(\mathbf{x})\}
$ is the {\it Lagrangian dual function} of problem \eqref{eq:program-objective}-\eqref{eq:program-set-constraint}.\end{Assumption}

The existence of Lagrange multipliers attaining strong duality is a mild assumption. 
For convex programs, it is implied by the existence of a vector $\mathbf{s} \in \mathcal{X}$ such that
$g_k(\mathbf{s}) < 0$ for all $k \in \{1, \ldots, m\}$, called the \emph{Slater  condition}  \cite{book_NonlinearProgrammingTA,book_ConvexOptimization}.  

\subsection{Large Scale Convex Programs}
In general, convex program \eqref{eq:program-objective}-\eqref{eq:program-set-constraint} can be solved via interior point methods  (or other Newton type methods) which involve the computation of Hessians and matrix inversions at each iteration.  The associated computation complexity and memory space complexity at each iteration is between $O(n^{2})$ and $O(n^{3}$), which is prohibitive when $n$ is extremely large. For example, if $n=10^{5}$ and each float point number uses $4$ bytes, then $40$ Gbytes of memory space is required even to save the Hessian at each iteration.  Thus, large scale convex programs are usually solved by gradient based methods or decomposition based methods.

\subsection{The Primal-Dual Subgradient Method, aka, Arrow-Hurwicz-Uzawa Subgradient Method}
The primal-dual subgradient method applied to convex program \eqref{eq:program-objective}-\eqref{eq:program-set-constraint} is described in Algorithm \ref{alg:primal-dual-subgradient}. The updates of $\mathbf{x}(t)$ and $\boldsymbol{\lambda}(t)$ only involve the computation of gradient and simple projection operations, which are much simpler than the computation of Hessians and matrix inversions for extremely large $n$.  Thus, compared with the interior point methods, the primal-dual subgradient algorithm has lower complexity computations at each iteration and hence is more suitable to large scale convex programs.  However, the convergence rate of Algorithm \ref{alg:primal-dual-subgradient} is only $O(1/\sqrt{t})$, where $t$ is the number of iterations \cite{Nedic09_PrimalDualSubgradient}. 

\begin{algorithm} 
\caption{The Primal-Dual Subgradient Algorithm}
\label{alg:primal-dual-subgradient}
Let $c>0$ be a constant step size. Choose any $\mathbf{x}(0) \in \mathcal{X}$. Initialize Lagrangian multipliers $\lambda_{k}(0) = 0, \forall k\in\{1,2,\ldots, m\}$.  At each iteration $t\in\{1,2,\ldots\}$, observe $\mathbf{x}(t-1)$ and $\boldsymbol{\lambda}(t-1)$ and do the following:
\begin{itemize}
\item  Choose $\mathbf{x}(t)$ via 
\begin{align*}
\mathbf{x}(t)  =\mathcal{P}_{\mathcal{X}} \left[ \mathbf{x}(t-1) - c \sum_{k=1}^{m} \lambda_{k}(t-1) \nabla g_{k}(\mathbf{x}(t-1))\right] ,
\end{align*}
where $\mathcal{P}_{\mathcal{X}}[\cdot]$ is the projection onto convex set $\mathcal{X}$.
\item Update Lagrangian multipliers $\boldsymbol{\lambda}(t)$ via 
\begin{align*}
\lambda_{k}(t) = \left[ \lambda_{k}(t-1) + c g_{k}(\mathbf{x}(t-1))\right]_{0}^{\lambda_{k}^{\max}}, \forall k\in\{1,2,\ldots, m\},
\end{align*}
where $\lambda_{k}^{\max} > \lambda_{k}^{\ast}$ and $[\cdot ]_{0}^{\max}$ is the projection onto interval $[0,\lambda_{k}^{\max}]$.
\item Update the running averages $\overline{\mathbf{x}}(t)$ via
\begin{align*}
\overline{\mathbf{x}}(t+1) = \frac{1}{t} \sum_{\tau=0}^{t} \mathbf{x}(\tau) = \overline{\mathbf{x}}(t) \frac{t}{t+1} + \mathbf{x}(t) \frac{1}{t+1}
\end{align*}
\end{itemize}
\end{algorithm}

\subsection{The Lagrangian Dual Type Method}
The classical dual subgradient algorithm is a Lagrangian dual type iterative method that approaches optimality for strictly convex programs \cite{book_NonlinearProgramming_Bertsekas}. A modification of the classical dual subgradient algorithm that averages the resulting sequence of primal estimates can solve general convex programs and has the $O(1/\sqrt{t})$ convergence rate \cite{Neely05DCDIS,Nedic09,Neely14Arxiv_ConvergenceTime}. The dual subgradient algorithm with averaged primals is suitable to large scale convex programs because the updates of each component $x_{i}(t)$ are independent and parallel if functions $f(\mathbf{x})$ and $g_{k}(\mathbf{x})$ in convex program \eqref{eq:program-objective}-\eqref{eq:program-set-constraint} are separable with respect to each component (or block) of $\mathbf{x}$, e.g., $f(\mathbf{x}) = \sum_{i=1}^{n} f_{i} (x_{i})$ and $g_{k}(\mathbf{x}) = \sum_{i=1}^{n} g_{k,i}(x_{i})$.

Recently, a new Lagrangian dual type algorithm with  convergence rate $O(1/t)$ for general convex programs is proposed in \cite{YuNeely15ArXivGeneralConvex}.  This algorithm can solve convex program \eqref{eq:program-objective}-\eqref{eq:program-set-constraint} following the steps described in Algorithm \ref{alg:general-alg}.

\begin{algorithm} 
\caption{Algorithm 1 in \cite{YuNeely15ArXivGeneralConvex}}
\label{alg:general-alg}
Let $\alpha>0$ be a constant parameter. Choose any $\mathbf{x}(-1) \in \mathcal{X}$. Initialize virtual queues $Q_{k}(0) = \max\{0, -g_{k}(\mathbf{x}(-1))\} , \forall k\in\{1,2,\ldots, m\}$. At each iteration $t\in\{0,1,2,\ldots\}$, observe $\mathbf{x}(t-1)$ and $\mathbf{Q}(t)$ and do the following:
\begin{itemize}
\item  Choose $\mathbf{x}(t)$ as 
\begin{align*}
\mathbf{x}(t)  =\argmin_{\mathbf{x}\in \mathcal{X}} \Big\{ f(\mathbf{x})  + [\mathbf{Q}(t) + \mathbf{g}(\mathbf{x}(t-1))]^T\mathbf{g}(\mathbf{x}) +  \alpha \Vert \mathbf{x} - \mathbf{x}(t-1)\Vert^{2}\Big\}.
\end{align*}

\item Update virtual queue vector $\mathbf{Q}(t)$ via  \[Q_{k}(t+1) = \max\{-g_{k}(\mathbf{x}(t)), Q_{k}(t) + g_{k}(\mathbf{x}(t))\}, \forall k\in\{1,2,\ldots, m\}.\]
\item Update the running averages $\overline{\mathbf{x}}(t)$ via \[\overline{\mathbf{x}}(t+1) = \frac{1}{t+1}\sum_{\tau=0}^{t} \mathbf{x}(\tau) = \overline{\mathbf{x}}(t) \frac{t}{t+1} + \mathbf{x}(t) \frac{1}{t+1}.\]
\end{itemize}
\end{algorithm}

Similar to the dual subgradient algorithm with averaged primals, Algorithm \ref{alg:general-alg} can decompose the updates of $\mathbf{x}(t)$ into smaller independent subproblems if functions $f(\mathbf{x})$ and $g_{k}(\mathbf{x})$ are separable. Moreover, Algorithm \ref{alg:general-alg} has convergence rate $O(1/t)$, which is faster than the primal-dual subgradient algorithm or the dual subgradient algorithm with averaged primals.

However, in the case $f(\mathbf{x})$ or $g_{k}(\mathbf{x})$ are not separable,  each update of $\mathbf{x}(t)$ requires to solve a set constrained convex program. If the dimension $n$ is large, such a set constrained convex program should be solved via a gradient based method instead of a Newton method. However, the gradient based method for set constrained convex programs is an iterative technique and involves at least one projection operation at each iteration. For instance, to obtain an $\epsilon$-approximate solution to the set constrained convex program, the projected gradient method requires $O(1/\epsilon)$ iterations and Nesterov's fast gradient method requires $O(1/\sqrt{\epsilon})$ iterations \cite{book_ConvexOpt_Nesterov}.

\subsection{New Algorithm}
Consider large scale convex programs with non-separable $f(\mathbf{x})$ or $g_{k}(\mathbf{x})$, e.g., $f(\mathbf{x}) = \Vert \mathbf{A} \mathbf{x} - \mathbf{b}\Vert^{2}$. In this case,  Algorithm \ref{alg:primal-dual-subgradient} has convergence rate $O(1/\sqrt{t})$ using low complexity iterations; while Algorithm \ref{alg:general-alg} has convergence rate $O(1/t)$  using high complexity iterations.  

This paper proposes a new algorithm described in Algorithm \ref{alg:new-alg} which combines the advantages of Algorithm \ref{alg:primal-dual-subgradient} and Algorithm \ref{alg:general-alg}.  The new algorithm modifies  Algorithm \ref{alg:general-alg} by changing the update of $\mathbf{x}(t)$ from a complicated minimization problem to a simple projection operation.  Meanwhile, the convergence rate $O(1/t)$ of Algorithm \ref{alg:general-alg} is preserved in the new algorithm.
     
\begin{algorithm} 
\caption{New Algorithm}
\label{alg:new-alg}
Let $\gamma >0$ be a constant step size. Choose any $\mathbf{x}(-1) \in \mathcal{X}$. Initialize virtual queues $Q_{k}(0) = \max\{0, -g_{k}(\mathbf{x}(-1))\} , \forall k\in\{1,2,\ldots, m\}$. At each iteration $t\in\{0,1,2,\ldots\}$, observe $\mathbf{x}(t-1)$ and $\mathbf{Q}(t)$ and do the following:
\begin{itemize}
\item  Define $\mathbf{d}(t) = \nabla f(\mathbf{x}(t-1)) + \sum_{k=1}^{m} [Q_{k}(t) + g_{k}(\mathbf{x}(t-1))] \nabla g_{k}(\mathbf{x}(t-1))$, which is the gradient of function $\phi(\mathbf{x}) =  f(\mathbf{x}) +  [\mathbf{Q}(t) + \mathbf{g}(\mathbf{x}(t-1))]^{T} \mathbf{g}(\mathbf{x})$ at point $\mathbf{x} = \mathbf{x}(t-1)$. Choose $\mathbf{x}(t)$ as 
\begin{align*}
\mathbf{x}(t)  =\mathcal{P}_{\mathcal{X}}\left[ \mathbf{x}(t-1) -\gamma \mathbf{d}(t) \right],
\end{align*}
where $\mathcal{P}_{\mathcal{X}}[\cdot]$ is the projection onto convex set $\mathcal{X}$.
\item Update virtual queue vector $\mathbf{Q}(t)$ via  \[Q_{k}(t+1) = \max\{-g_{k}(\mathbf{x}(t)), Q_{k}(t) + g_{k}(\mathbf{x}(t))\}, \forall k\in\{1,2,\ldots, m\}.\]
\item Update the running averages $\overline{\mathbf{x}}(t)$ via \[\overline{\mathbf{x}}(t+1) = \frac{1}{t+1}\sum_{\tau=0}^{t} \mathbf{x}(\tau) = \overline{\mathbf{x}}(t) \frac{t}{t+1} + \mathbf{x}(t) \frac{1}{t+1}.\]
\end{itemize}
\end{algorithm}

\section{Preliminaries and Basis Analysis}

This section presents useful preliminaries on convex analysis and important facts of Algorithm \ref{alg:new-alg}.

\subsection{Preliminaries}

\begin{Def}[Lipschitz Continuity] \label{def:Lipschitz-continuous}
Let $\mathcal{X} \subseteq \mathbb{R}^n$ be a convex set. Function $h: \mathcal{X}\rightarrow \mathbb{R}^m$ is said to be Lipschitz continuous  on $\mathcal{X}$ with modulus $L$ if there exists $L> 0$ such that $\Vert h(\mathbf{y}) - h(\mathbf{x}) \Vert \leq L \Vert\mathbf{y} - \mathbf{x}\Vert$  for all $ \mathbf{x}, \mathbf{y} \in \mathcal{X}$. 
\end{Def}

\begin{Def}[Smooth Functions]
Let $\mathcal{X} \subseteq \mathbb{R}^n$ and function $h(\mathbf{x})$ be continuously differentiable on $\mathcal{X}$. Function $h(\mathbf{x})$ is said to be smooth on $\mathcal{X}$ with modulus $L$ if $\nabla h(\mathbf{x})$ is Lipschitz continuous on $\mathcal{X}$ with modulus $L$.
\end{Def}

Note that linear function $h(\mathbf{x}) = \mathbf{a}^T \mathbf{x}$ is smooth with modulus $0$.  If a function $h(\mathbf{x})$ is smooth with modulus $L$, then $c h(\mathbf{x})$ is smooth with modulus $cL$ for any constant  $c>0$.

\begin{Lem}[Descent Lemma, Proposition A.24 in \cite{book_NonlinearProgramming_Bertsekas}] \label{lm:descent-lemma}
If $h$ is smooth on $\mathcal{X}$ with modulus $L$, then for any $\mathbf{x}, \mathbf{y} \in \mathcal{X}$
\begin{align*}
 \nabla h(\mathbf{x})^T (\mathbf{y} - \mathbf{x}) - \frac{L}{2} || \mathbf{y} - \mathbf{x}||^2
 \leq h(\mathbf{y}) - h(\mathbf{x}) \leq \nabla h(\mathbf{x})^T (\mathbf{y} - \mathbf{x}) + \frac{L}{2} || \mathbf{y} - \mathbf{x}||^2.
\end{align*}
\end{Lem}
\begin{Def}[Strongly Convex Functions]
 Let $\mathcal{X} \subseteq \mathbb{R}^n$ be a convex set. Function $h$ is said to be strongly convex on $\mathcal{X}$ with modulus $\alpha$ if there exists a constant $\alpha>0$ such that $h(\mathbf{x}) - \frac{1}{2} \alpha \Vert \mathbf{x} \Vert^2$ is convex on $\mathcal{X}$.
\end{Def}

By the definition of strongly convex functions, it is easy to show that if $h(\mathbf{x})$ is convex and $\alpha>0$, then $h(\mathbf{x}) + \alpha \Vert \mathbf{x} - \mathbf{x}_0\Vert^2$ is strongly convex with modulus $2\alpha$ for any constant $\mathbf{x}_0$.

\begin{Lem}[Theorem 6.1.2 in \cite{book_FundamentalConvexAnalysis}] \label{lm:strong-convex}
Let $h(\mathbf{x})$ be strongly convex on $\mathcal{X}$ with modulus $\alpha$. Let $\partial h(\mathbf{x})$ be the set of all subgradients of $h$ at point $\mathbf{x}$. Then $h(\mathbf{y}) \geq h(\mathbf{x}) + \mathbf{d}^T (\mathbf{y} - \mathbf{x}) + \frac{\alpha}{2}\Vert \mathbf{y} - \mathbf{x} \Vert^2$ for all $\mathbf{x}, \mathbf{y}\in \mathcal{X}$ and all $\mathbf{d}\in \partial f(\mathbf{x})$.
\end{Lem}

\begin{Lem}[Proposition B.24 (f) in \cite{book_NonlinearProgramming_Bertsekas}] \label{lm:first-order-optimality}
Let $\mathcal{X}\subseteq \mathbb{R}^n$ be a convex set. Let function $h$ be convex on $\mathcal{X}$ and $\mathbf{x}^{opt}$ be the global minimum of $h$ on $\mathcal{X}$.  Let $\partial h(\mathbf{x})$ be the set of all subgradients of $h$ at point $\mathbf{x}$. Then, there exists $\mathbf{d} \in \partial h(\mathbf{x}^{opt})$ such that $\mathbf{d}^T(\mathbf{x} - \mathbf{x}^{opt}) \geq 0$ for all $\mathbf{x}\in \mathcal{X}$. 
\end{Lem}

\begin{Cor} \label{cor:strong-convex-quadratic-optimality}
Let $\mathcal{X} \subseteq \mathbb{R}^{n}$ be a convex set. Let function $h$ be strongly convex on $\mathcal{X}$ with modulus $\alpha$ and $\mathbf{x}^{opt}$ be the global minimum of $h$ on $\mathcal{X}$. Then, $h(\mathbf{x}^{opt}) \leq h(\mathbf{x}) - \frac{\alpha}{2} \Vert \mathbf{x}^{opt} - \mathbf{x}\Vert^{2}$ for all $\mathbf{x}\in \mathcal{X}$.
\end{Cor}
\begin{IEEEproof}
Fix $\mathbf{x}\in \mathcal{X}$. By Lemma \ref{lm:first-order-optimality}, there exists $\mathbf{d} \in \partial h(\mathbf{x}^{opt})$ such that $\mathbf{d}^T(\mathbf{x} - \mathbf{x}^{opt}) \geq 0$. By Lemma \ref{lm:strong-convex}, we also have
\begin{align*}
h(\mathbf{x}) &\geq h(\mathbf{x}^{opt}) +  \mathbf{d}^{T} (\mathbf{x} - \mathbf{x}^{opt}) + \frac{\alpha}{2} \Vert  \mathbf{x} -\mathbf{x}^{opt} \Vert^{2}\\
&\overset{(a)}{\geq} h(\mathbf{x}^{opt}) + \frac{\alpha}{2} \Vert  \mathbf{x} -\mathbf{x}^{opt} \Vert^{2}
\end{align*}
where $(a)$ follows from the fact that $\mathbf{d}^T(\mathbf{x} - \mathbf{x}^{opt}) \geq 0$.
\end{IEEEproof}

\subsection{Properties of the Virtual Queues}

The following preliminary results (Lemmas \ref{lm:virtual-queue}-\ref{lm:obj-diff-bound-from-strong-duality}) are similar to those proven for a different algorithm in our prior technical report \cite{YuNeely15ArXivGeneralConvex}.  The corresponding proofs for the current algorithm are similar. For convenience to the reader, we include proofs for these lemmas.

\begin{Lem} [Lemma 3 in \cite{YuNeely15ArXivGeneralConvex}]\label{lm:virtual-queue} In Algorithm \ref{alg:new-alg}, we have
\begin{enumerate}
\item At each iteration $t\in\{0,1,2,\ldots\}$, $Q_k(t)\geq 0$ for all $k\in\{1,2,\ldots,m\}$.
\item At each iteration $t\in\{0,1,2,\ldots\}$, $Q_{k}(t) + g_{k}(\mathbf{x}(t-1))\geq 0$ for all $k\in\{1,2\ldots, m\}$.
\item At iteration $t=0$, $\Vert \mathbf{Q}(0)\Vert^2 \leq \Vert \mathbf{g}(\mathbf{x}(-1))\Vert^2$. At each iteration $t\in\{1,2,\ldots\}$,  $\Vert \mathbf{Q}(t)\Vert^2 \geq \Vert \mathbf{g}(\mathbf{x}(t-1))\Vert^2$.
\end{enumerate}
\end{Lem}
\begin{IEEEproof} 
\begin{enumerate}
\item Fix $k\in\{1,2,\ldots, m\}$. Note that $Q_{k}(0) + g_{k}(\mathbf{x}(-1))\geq 0$ by initialization rule $Q_{k}(0) = \max\{0, -g_{k}(\mathbf{x}(-1))\}$. Assume that $Q_k(t) \geq 0$ and consider time $t+1$. If $g_k(\mathbf{x}(t)) \geq 0$, then $Q_k(t+1) = \max\{-g_{k}(\mathbf{x}(t)), Q_{k}(t) + g_{k}(\mathbf{x}(t))\} \geq Q_{k}(t) + g_{k}(\mathbf{x}(t))\geq 0$. If $g_k(\mathbf{x}(t))<0$, then $Q_k(t+1) = \max\{-g_{k}(\mathbf{x}(t)), Q_{k}(t) + g_{k}(\mathbf{x}(t))\}\geq -g_{k}(\mathbf{x}(t)) > 0$.  Thus, $Q_k(t+1) \geq 0$. The result follows by induction.
\item Fix $k\in\{1,2,\ldots, m\}$. Note that $Q_{k}(0) + g_{k}(\mathbf{x}(-1))\geq 0$ by initialization rule $Q_{k}(0) = \max\{0, -g_{k}(\mathbf{x}(-1))\} \geq -g_{k}(\mathbf{x}(-1))$. For $t\geq 1$, by the virtual update equation, we have $Q_{k}(t) = \max\{-g_{k}(\mathbf{x}(t-1)), Q_{k}(t-1) + g_{k}(\mathbf{x}(t-1))\}\geq -g_{k}(\mathbf{x}(t-1))$, which implies that $Q_{k}(t) + g_{k}(\mathbf{x}(t-1)) \geq 0$.
\item 
\begin{itemize}
\item For $t=0$. Fix $k\in\{1,2,\ldots, m\}$. Consider the cases $g_k(\mathbf{x}(-1))\geq 0$ and $g_k(\mathbf{x}(-1))<0$ separately. If $g_k(\mathbf{x}(-1)) \geq 0$, then $ Q_k(0) = \max\{0,-g_{k}(\mathbf{x}(-1))\} =0 \leq g_k(\mathbf{x}(-1))$. If $g_k(\mathbf{x}(-1)) < 0$, then $Q_k(0) = \max\{0, -g_{k}(\mathbf{x}(t-1))\} = -g_k(\mathbf{x}(-1))>0$. Thus, in both cases, we have $\vert Q_k(0) \vert \leq \vert g_k(\mathbf{x}(-1))\vert$.  Squaring both sides and summing over $k\in\{1,2,\ldots,m\}$ yields $\Vert \mathbf{Q}(0)\Vert^2 \leq \Vert \mathbf{g}(\mathbf{x}(-1))\Vert^2$.  
\item For $t\geq 1$. Fix $k\in\{1,2,\ldots, m\}$. Consider the cases $g_k(\mathbf{x}(t-1))\geq 0$ and $g_k(\mathbf{x}(t-1))<0$ separately. If $g_k(\mathbf{x}(t-1)) \geq 0$, then $Q_k(t) = \max\{-g_{k}(\mathbf{x}(t-1)), Q_{k}(t-1) + g_{k}(\mathbf{x}(t-1))\} \geq Q_{k}(t-1) + g_{k}(\mathbf{x}(t-1)) \overset{(a)}{\geq} g_k(\mathbf{x}(t-1)) \geq 0$ where (a) follows from part (1). If $g_k(\mathbf{x}(t-1)) < 0$, then $Q_k(t) = \max\{-g_{k}(\mathbf{x}(t-1)), Q_{k}(t-1) + g_{k}(\mathbf{x}(t-1))\} \geq -g_k(\mathbf{x}(t-1)) > 0$. Thus, in both cases, we have $\vert Q_k(t) \vert \geq \vert g_k(\mathbf{x}(t-1))\vert$.  Squaring both sides and summing over $k\in\{1,2,\ldots,m\}$ yields $\Vert \mathbf{Q}(t)\Vert^2 \geq \Vert \mathbf{g}(\mathbf{x}(t-1))\Vert^2$.
\end{itemize}
\end{enumerate}
\end{IEEEproof}

\begin{Lem}[Lemma 7 in \cite{YuNeely15ArXivGeneralConvex}]\label{lm:queue-constraint-inequality}
Let $\mathbf{Q}(t), t\in\{0,1,\ldots\}$ be the sequence generated by Algorithm \ref{alg:new-alg}.  
For any $t\geq 1$, 
\[ Q_k(t) \geq   \displaystyle{\sum_{\tau=0}^{t-1} g_k(\mathbf{x}(\tau))}, \forall k\in\{1,2,\ldots,m\}. \]
\end{Lem}
\begin{IEEEproof}
Fix $k\in\{1,2,\ldots,m\}$ and $t \geq 1$.  For any $\tau \in \{0, \ldots, t-1\}$ the update rule of Algorithm \ref{alg:new-alg} gives: 
\begin{align*}
Q_k(\tau+1) &= \max\{-g_{k}(\mathbf{x}(\tau)), Q_k(\tau)+g_k(\mathbf{x}(\tau))\} \\
&\geq Q_k(\tau) + g_k(\mathbf{x}(\tau)). 
\end{align*}
Hence, $Q_k(\tau+1) - Q_k(\tau) \geq g_k(\mathbf{x}(\tau))$.  
Summing over $\tau \in \{0, \ldots, t-1\}$ and using $Q_k(0)\geq 0$ gives the result. 
\end{IEEEproof}

\subsection{Properties of the Drift}

Let $\mathbf{Q}(t) = \big[ Q_1(t), \ldots, Q_m(t)\big]^T$ be the vector of virtual queue backlogs.  Define  $L(t) = \frac{1}{2} \Vert \mathbf{Q}(t)\Vert^2$. The function $L(t)$ shall be called a \emph{Lyapunov function}. Define the {Lyapunov drift} as 
\begin{align}
\Delta (t) = L(t+1) - L(t) = \frac{1}{2} [ \Vert \mathbf{Q}(t+1)\Vert^{2} - \Vert \mathbf{Q}(t)\Vert^{2}]. \label{eq:def-drift}
\end{align}

\begin{Lem} [Lemma 4 in \cite{YuNeely15ArXivGeneralConvex}]\label{lm:drift} At each iteration $t\in\{0,1,2,\ldots\}$ in Algorithm \ref{alg:new-alg}, an upper bound of the Lyapunov drift is given by
\begin{align}
\Delta(t) \leq \mathbf{Q}^T(t) \mathbf{g}(\mathbf{x}(t))  +\Vert \mathbf{g}(\mathbf{x}(t))\Vert^2.  \label{eq:drift}
\end{align}
\end{Lem}

\begin{IEEEproof} 
The virtual queue update equations $Q_k(t+1) = \max\{-g_k(\mathbf{x}(t)), Q_k(t) + g_k(\mathbf{x}(t))\}, \forall k\in \{1,2,\ldots,m\}$ can be rewritten as
\begin{align}
Q_k(t+1) = Q_k(t) + \tilde{g}_k(\mathbf{x}(t)), \forall k\in \{1,2,\ldots,m\}, \label{eq:modified-virtual-queue}
\end{align} 
where 
\begin{equation*}
\tilde{g}_k(\mathbf{x}(t)) = \left \{ \begin{array}{cl}  g_k(\mathbf{x}(t)), & \text{if}~Q_k(t) + g_k(\mathbf{x}(t)) \geq -g_k(\mathbf{x}(t))\\
-Q_k(t) - g_k(\mathbf{x}(t)),  & \text{else} \end{array} \right. \forall k.
\end{equation*}

Fix $k\in\{1,2,\ldots, m\}$. Squaring both sides of \eqref{eq:modified-virtual-queue} and dividing by $2$ yield: 
\begin{align*}
&\frac{1}{2}[Q_k (t+1) ]^2 \\
= &\frac{1}{2}[Q_k(t)]^2 + \frac{1}{2}[\tilde{g}_k(\mathbf{x}(t))]^2 +  Q_k (t) \tilde{g}_k(\mathbf{x}(t)) \\
= &\frac{1}{2}[Q_k(t)]^2 + \frac{1}{2}[\tilde{g}_k(\mathbf{x}(t))]^2 +  Q_k (t) g_k(\mathbf{x}(t)) +  Q_k (t)[\tilde{g}_k(\mathbf{x}(t)) -g_k(\mathbf{x}(t))]  \\
\overset{(a)}{=}& \frac{1}{2}[Q_k(t)]^2 + \frac{1}{2}[\tilde{g}_k(\mathbf{x}(t))]^2 + Q_k (t) g_k(\mathbf{x}(t)) \\& - [ \tilde{g}_k (\mathbf{x}(t)) + g_k(\mathbf{x}(t))] [ \tilde{g}_k(\mathbf{x}(t)) -g_k(\mathbf{x}(t)) ]\\
=& \frac{1}{2}[Q_k(t)]^2 - \frac{1}{2}[\tilde{g}_k(\mathbf{x}(t))]^2 + Q_k (t) g_k(\mathbf{x}(t))  + [g_k(\mathbf{x}(t))]^2 \\
\leq &\frac{1}{2}[Q_k(t)]^2 + Q_k (t) g_k(\mathbf{x}(t))  + [g_k(\mathbf{x}(t))]^2,
\end{align*}
where $(a)$ follows from the fact that $Q_k (t) [ \tilde{g}_k (\mathbf{x}(t)) -g_k(\mathbf{x}(t)) ] = - [\tilde{g}_k (\mathbf{x}(t)) +g_k(\mathbf{x}(t))] \cdot [ \tilde{g}_k (\mathbf{x}(t)) -g_k(\mathbf{x}(t)) ]$, which can be shown by considering $\tilde{g}_k (\mathbf{x}(t)) = g_k(\mathbf{x}(t))$ and $\tilde{g}_k (\mathbf{x}(t)) \neq g_k(\mathbf{x}(t))$.
Summing over $k\in\{1,2,\ldots,m\}$ yields 
\begin{align*}
\frac{1}{2}\Vert\mathbf{Q}(t+1)\Vert^{2} \leq \frac{1}{2}\Vert\mathbf{Q}(t)\Vert^{2}  + \mathbf{Q}^T(t ) \mathbf{g}(\mathbf{x}(t)) + \Vert \mathbf{g}(\mathbf{x}(t))\Vert^2. 
\end{align*}

Rearranging the terms yields the desired result.
\end{IEEEproof}

\subsection{Properties from Strong Duality}

\begin{Lem}[Lemma 8 in \cite{YuNeely15ArXivGeneralConvex}] \label{lm:obj-diff-bound-from-strong-duality}
Let $\mathbf{x}^{\ast}$ be an optimal solution of problem \eqref{eq:program-objective}-\eqref{eq:program-set-constraint} and $\boldsymbol{\lambda}^\ast$ be a Lagrange multiplier vector satisfying Assumption \ref{as:strong-duality}. Let $\mathbf{x}(t), \mathbf{Q}(t), t\in\{0,1,\ldots\}$ be sequences generated by Algorithm \ref{alg:new-alg}. Then,
\begin{align*}
\sum_{\tau=0}^{t-1} f(\mathbf{x}(\tau)) \geq t f(\mathbf{x}^\ast) -  \Vert \boldsymbol{\lambda}^\ast\Vert \Vert \mathbf{Q}(t)\Vert, \quad \forall t\geq 1. 
\end{align*}
\end{Lem}

\begin{IEEEproof}
Define Lagrangian dual function $q(\boldsymbol{\lambda}) = \min\limits_{\mathbf{x}\in \mathcal{X}}\{f(\mathbf{x})+ \sum_{k=1}^m \lambda_k g_k(\mathbf{x})\}$. For all $\tau\in\{0,1,\ldots\}$, by Assumption \ref{as:strong-duality}, we have
\begin{align*}
f(\mathbf{x}^{\ast}) = q(\boldsymbol{\lambda}^{\ast}) \overset{(a)}{\leq}  f(\mathbf{x}(\tau)) + \sum_{k=1}^m \lambda_k^\ast g_k(\mathbf{x}(\tau)),
\end{align*} 
where (a) follows the definition of $q(\boldsymbol{\lambda}^{\ast})$.

Thus, we have $ f(\mathbf{x}(\tau))  \geq f(\mathbf{x}^\ast) -\sum_{k=1}^m \lambda_k^\ast g_k(\mathbf{x}(\tau)),\forall \tau \in\{0,1,\ldots\}$. Summing over $\tau\in \{0,1,\ldots, t\}$ yields 
\begin{align}
\sum_{\tau=0}^{t-1} f(\mathbf{x}(\tau)) \geq &t f(\mathbf{x}^\ast) -  \sum_{\tau=0}^{t-1} \sum_{k=1}^m\lambda_k^\ast g_k(\mathbf{x}(\tau))\nonumber \\
=&t f(\mathbf{x}^\ast)  - \sum_{k=1}^m  \lambda_k^\ast \Big[\sum_{\tau=0}^{t-1}g_k(\mathbf{x}(\tau))\Big] \nonumber\\
\overset{(a)}{\geq}& t f(\mathbf{x}^\ast)  -\sum_{k=1}^m  \lambda_k^\ast Q_k(t) \nonumber\\
\overset{(b)}{\geq}& t f(\mathbf{x}^\ast)  -\Vert \boldsymbol{\lambda}^\ast\Vert \Vert \mathbf{Q}(t)\Vert, \nonumber
\end{align}
where $(a)$ follows from Lemma \ref{lm:queue-constraint-inequality} and the fact that $\lambda_k^\ast \geq 0, \forall k\in\{1,2,\ldots, m\}$; and $(b)$ follows from the Cauchy-Schwartz inequality. \end{IEEEproof}

\section{Convergence Rate Analysis of Algorithm \ref{alg:new-alg} }

This section analyzes the convergence rate of Algorithm \ref{alg:new-alg} for problem \eqref{eq:program-objective}-\eqref{eq:program-set-constraint}.

\subsection{An Upper Bound of the Drift-Plus-Penalty Expression}

\begin{Lem}\label{lm:dpp-bound}
Let $\mathbf{x}^{\ast}$ be an optimal solution of problem \eqref{eq:program-objective}-\eqref{eq:program-set-constraint}. For all $t\geq 0$ in Algorithm \ref{alg:new-alg}, we have
\begin{align*}
&\Delta(t) + f(\mathbf{x}(t)) \\
\leq &f(\mathbf{x}^{\ast}) + \frac{1}{2\gamma} [\Vert \mathbf{x} ^{\ast}- \mathbf{x}(t-1)\Vert^{2} - \Vert \mathbf{x}^{\ast} - \mathbf{x}(t)\Vert^{2}] +\frac{1}{2} [ \Vert \mathbf{g}(\mathbf{x}(t))\Vert^{2} -  \Vert \mathbf{g}(\mathbf{x}(t-1))\Vert^{2} ] \\ &+ \frac{1}{2} \big[ \beta^{2} + L_{f} + \Vert \mathbf{Q}(t) \Vert \Vert \mathbf{L}_{\mathbf{g}}\Vert + C \Vert \mathbf{L}_{\mathbf{g}} \Vert-\frac{1}{\gamma} \big] \Vert \mathbf{x}(t) - \mathbf{x}(t-1)\Vert^{2},
\end{align*}
where $\beta, L_{f}$, $\mathbf{L}_{\mathbf{g}}$ and $C$ are defined in Assumption \ref{as:basic}.
\end{Lem}

\begin{IEEEproof} 
Fix $t\geq 0$. The projection operator can be reinterpreted as an optimization problem as follows:
\begin{align}
&\mathbf{x}(t) =  \mathcal{P}_{\mathcal{X}} \left[ \mathbf{x}(t-1) - \gamma \mathbf{d}(t)\right] \nonumber\\
\overset{(a)}{\Leftrightarrow}\quad& \mathbf{x}(t)= \argmin_{\mathbf{x}\in \mathcal{X}} \left[ \big\Vert \mathbf{x} - [\mathbf{x}(t-1)- \gamma \mathbf{d}(t)] \big\Vert^{2}\right] \nonumber\\
\Leftrightarrow\quad& \mathbf{x}(t) = \argmin_{\mathbf{x}\in \mathcal{X}} \left[ \Vert \mathbf{x} -\mathbf{x}(t-1) \Vert^{2} + 2\gamma \mathbf{d}^{T}(t)[\mathbf{x} -\mathbf{x}(t-1)] + \gamma^2 \Vert \mathbf{d}(t) \Vert^{2}\right] \nonumber\\
\overset{(b)}{\Leftrightarrow} \quad&  \mathbf{x}(t)= \argmin_{\mathbf{x}\in \mathcal{X}} \left[ f(\mathbf{x}(t-1))+  \sum_{k=1}^m [ Q_{k}(t) + g_k (\mathbf{x}(t-1))] g_{k}(\mathbf{x}(t-1))+   \mathbf{d}^{T}(t)[\mathbf{x} -\mathbf{x}(t-1)] +  \frac{1}{2\gamma}\Vert \mathbf{x} -\mathbf{x}(t-1) \Vert^{2} \right] \nonumber\\
\overset{(c)}{\Leftrightarrow} \quad& \mathbf{x}(t)= \argmin_{\mathbf{x}\in \mathcal{X}} \left[ \phi(\mathbf{x}(t-1)) +  \nabla^{T} \phi(\mathbf{x}(t-1)) [\mathbf{x} - \mathbf{x}(t-1)] + \frac{1}{2\gamma} \Vert \mathbf{x} - \mathbf{x}(t-1)\Vert^{2}\right],\label{eq:pf-lm-primal-update-eq1}
\end{align}
where (a) follows from the definition of the projection on to a convex set; (b) follows from the fact the minimizing solution does not change when we remove constant term $\gamma^2 \Vert \mathbf{d}(t)\Vert^{2}$, multiply positive constant $\frac{1}{2\gamma}$ and add constant term $f(\mathbf{x}(t-1))+ [ \mathbf{Q}(t) + \mathbf{g} (\mathbf{x}(t-1))]^{T} \mathbf{g}(\mathbf{x}(t-1)) $ in the objective function; and (c) follows by defining 
\begin{align}
\phi(\mathbf{x}) =  f(\mathbf{x}) +  [\mathbf{Q}(t) + \mathbf{g}(\mathbf{x}(t-1))]^{T} \mathbf{g}(\mathbf{x}). \label{eq:pf-dpp-bound-def-phi} 
\end{align}
Note that part 2 in Lemma \ref{lm:virtual-queue} implies that $\mathbf{Q}(t) + \mathbf{g}(\mathbf{x}(t-1))$ is component-wise nonnegative for all $k\in\{1,2,\ldots, m\}$. Hence, function $\phi(\mathbf{x}) =  f(\mathbf{x}) +  [\mathbf{Q}(t) + \mathbf{g}(\mathbf{x}(t-1))]^{T} \mathbf{g}(\mathbf{x})$ is convex with respect to $\mathbf{x}$ on $\mathcal{X}$.

Since $ \frac{1}{2\gamma} \Vert \mathbf{x} - \mathbf{x}(t-1)\Vert^{2}$ is strongly convex with respect to $\mathbf{x}$ with modulus $\frac{1}{\gamma}$, it follows that
\begin{align*}
\phi(\mathbf{x}(t-1)) +  \nabla^{T} \phi(\mathbf{x}(t-1)) [\mathbf{x} - \mathbf{x}(t-1)] + \frac{1}{2\gamma} \Vert \mathbf{x} - \mathbf{x}(t-1)\Vert^{2}
\end{align*}
 is strongly convex with respect to $\mathbf{x}$ with modulus $\frac{1}{\gamma}$.  

Since $\mathbf{x}(t)$ is chosen to minimize the above strongly convex function, by Corollary \ref{cor:strong-convex-quadratic-optimality},  we have
\begin{align}
 &\phi(\mathbf{x}(t-1)) +  \nabla^{T} \phi(\mathbf{x}(t-1)) [\mathbf{x}(t) - \mathbf{x}(t-1)] + \frac{1}{2\gamma} \Vert \mathbf{x}(t) - \mathbf{x}(t-1)\Vert^{2} \nonumber \\
 \leq & \phi(\mathbf{x}(t-1)) +  \nabla^{T} \phi(\mathbf{x}(t-1))  [\mathbf{x}^{\ast} - \mathbf{x}(t-1)] + \frac{1}{2\gamma} \Vert \mathbf{x}^{\ast} - \mathbf{x}(t-1)\Vert^{2} - \frac{1}{2\gamma} \Vert \mathbf{x}^{\ast} - \mathbf{x}(t)\Vert^{2} \nonumber \\
\overset{(a)}{\leq} &\phi(\mathbf{x}^{\ast}) + \frac{1}{2\gamma} [\Vert \mathbf{x}^{\ast} - \mathbf{x}(t-1)\Vert^{2} - \Vert \mathbf{x}^{\ast} - \mathbf{x}(t)\Vert^{2}] \nonumber \\
\overset{(b)}{=} & f(\mathbf{x}^{\ast}) + \underbrace{[\mathbf{Q}(t) + \mathbf{g}(\mathbf{x}(t-1))]^T\mathbf{g}(\mathbf{x}^\ast)}_{\leq0}  + \frac{1}{2\gamma} [\Vert \mathbf{x}^{\ast} - \mathbf{x}(t-1)\Vert^{2} -\Vert \mathbf{x}^{\ast} - \mathbf{x}(t)\Vert^{2} ]\nonumber \\
\overset{(c)}{\leq}& f(\mathbf{x}^{\ast}) + \frac{1}{2\gamma} [\Vert \mathbf{x}^{\ast} - \mathbf{x}(t-1)\Vert^{2} -\Vert \mathbf{x}^{\ast} - \mathbf{x}(t)\Vert^{2}], \label{eq:pf-dpp-bound-eq1}
\end{align}
where (a) follows from the fact that $\phi(\mathbf{x})$ is convex with respect to $\mathbf{x}$ on $\mathcal{X}$; (b) follows from the definition of function $\phi(\mathbf{x})$ in \eqref{eq:pf-dpp-bound-def-phi}; and (c) follows by using the fact that $g_{k}(\mathbf{x}^{\ast})\leq 0$ and $Q_{k}(t) + g_{k}(\mathbf{x}(t-1))\geq 0$ (i.e., part 2 in Lemma \ref{lm:virtual-queue}) for all $k\in\{1,2,\ldots,m\}$ to eliminate the term marked by an underbrace.

Recall that $f(\mathbf{x})$ is smooth on $\mathcal{X}$ with modulus $L_{f}$ by Assumption \ref{as:basic}. By Lemma \ref{lm:descent-lemma}, we have
\begin{align}
f(\mathbf{x}(t))  \leq f(\mathbf{x}(t-1)) + \nabla^{T} f(\mathbf{x}(t-1))[\mathbf{x}(t) - \mathbf{x}(t-1)] +\frac{L_{f}}{2} \Vert \mathbf{x}(t) - \mathbf{x}(t-1)\Vert^{2}. \label{eq:pf-dpp-bound-eq2}
\end{align}
Recall that each $g_{k}(\mathbf{x})$ is smooth on $\mathcal{X}$ with modulus $L_{g_{k}}$ by Assumption \ref{as:basic}. Thus, $[Q_{k}(t) + g_{k}(\mathbf{x}(t-1))] g_{k}(\mathbf{x})$ is smooth with modulus $[Q_{k}(t) + g_{k}(\mathbf{x}(t-1))] L_{g_{k}}$. By Lemma \ref{lm:descent-lemma}, we have
\begin{align}
&[Q_{k}(t) + g_{k}(\mathbf{x}(t-1)) ] g_{k}(\mathbf{x}(t)) \nonumber \\
\leq &[Q_{k}(t) + g_{k}(\mathbf{x}(t-1)) ] g_{k}(\mathbf{x}(t-1)) + [Q_{k}(t) + g_{k}(\mathbf{x}(t-1))] \nabla^{T} g_{k}(\mathbf{x}(t-1))  [\mathbf{x}(t) - \mathbf{x}(t-1)] \nonumber \\ &+  \frac{[Q_{k}(t) + g_{k}(\mathbf{x}(t-1))] L_{g_{k}}}{2} \Vert \mathbf{x}(t) - \mathbf{x}(t-1)\Vert^{2} . \label{eq:pf-dpp-bound-eq3}
\end{align}
Summing \eqref{eq:pf-dpp-bound-eq3} over $k\in\{1,2,\ldots, m\}$ yields
\begin{align}
&[\mathbf{Q}(t) + \mathbf{g}(\mathbf{x}(t-1))]^{T} \mathbf{g}(\mathbf{x}(t))\\
\leq &[\mathbf{Q}(t) + \mathbf{g}(\mathbf{x}(t-1))]^{T} \mathbf{g}(\mathbf{x}(t-1))+ \sum_{k=1}^{m}[Q_{k}(t) + g_{k}(\mathbf{x}(t-1))] \nabla^{T} g_{k}(\mathbf{x}(t-1))  [\mathbf{x}(t) - \mathbf{x}(t-1)]  \nonumber \\ &+  \frac{[\mathbf{Q}(t) + \mathbf{g}(\mathbf{x}(t-1))]^{T} \mathbf{L}_{\mathbf{g}}}{2} \Vert \mathbf{x}(t) - \mathbf{x}(t-1)\Vert^{2}. \label{eq:pf-dpp-bound-eq4}
\end{align}
Summing up \eqref{eq:pf-dpp-bound-eq2} and \eqref{eq:pf-dpp-bound-eq4} together yields
\begin{align}
&f(\mathbf{x}(t)) + [\mathbf{Q}(t) + \mathbf{g}(\mathbf{x}(t-1))]^{T} \mathbf{g}(\mathbf{x}(t))   \nonumber \\
\leq & f(\mathbf{x}(t-1)) + [\mathbf{Q}(t) + \mathbf{g}(\mathbf{x}(t-1))]^{T} \mathbf{g}(\mathbf{x}(t-1)) +\nabla^{T} f(\mathbf{x}(t-1)) [\mathbf{x}(t) - \mathbf{x}(t-1)] \nonumber \\ &+ \sum_{k=1}^{m}[Q_{k}(t) + g_{k}(\mathbf{x}(t-1))] \nabla^{T} g_{k}(\mathbf{x}(t-1))  [\mathbf{x}(t) - \mathbf{x}(t-1)] + \frac{L_{f} + [\mathbf{Q}(t) + \mathbf{g}(\mathbf{x}(t-1))]^{T} \mathbf{L}_{\mathbf{g}}}{2} \Vert \mathbf{x}(t) - \mathbf{x}(t-1)\Vert^{2} \nonumber \\
\overset{(a)}{=} & \phi(\mathbf{x}(t-1)) + \nabla^T \phi(\mathbf{x}(t-1)) [\mathbf{x}(t) - \mathbf{x}(t-1)] + \frac{L_{f} + [\mathbf{Q}(t) + \mathbf{g}(\mathbf{x}(t-1))]^{T} \mathbf{L}_{\mathbf{g}}}{2} \Vert \mathbf{x}(t) - \mathbf{x}(t-1)\Vert^{2},\label{eq:pf-dpp-bound-eq5}
\end{align}
where (a) follows from the definition of function $\phi(\mathbf{x})$ in \eqref{eq:pf-dpp-bound-def-phi}.

Substituting \eqref{eq:pf-dpp-bound-eq1} into \eqref{eq:pf-dpp-bound-eq5}  yields
 \begin{align}
&f(\mathbf{x}(t)) + [\mathbf{Q}(t) + \mathbf{g}(\mathbf{x}(t-1))]^{T} \mathbf{g}(\mathbf{x}(t))   \nonumber \\
\leq &f(\mathbf{x}^{\ast}) + \frac{1}{2\gamma}[\Vert \mathbf{x}^{\ast} - \mathbf{x}(t-1)\Vert^{2} - \Vert \mathbf{x}^{\ast} - \mathbf{x}(t)\Vert^{2}] + \frac{1}{2}\big[L_{f} + [\mathbf{Q}(t) + \mathbf{g}(\mathbf{x}(t-1))]^{T} \mathbf{L}_{\mathbf{g}} -\frac{1}{\gamma} \big] \Vert \mathbf{x}(t) - \mathbf{x}(t-1)\Vert^{2}. \label{eq:pf-dpp-bound-eq6}
\end{align}
 Note that $\mathbf{u}_{1}^{T} \mathbf{u}_{2} = \frac{1}{2} [\Vert  \mathbf{u}_{1}\Vert^{2}  + \Vert \mathbf{u}_{2}\Vert^{2} - \Vert \mathbf{u}_{1} - \mathbf{u}_{2}\Vert^{2} ]$ for any $\mathbf{u}_{1}, \mathbf{u}_{2}\in \mathbb{R}^{m}$. Thus, we have \begin{align}
[\mathbf{g}(\mathbf{x}(t-1))]^T \mathbf{g}(\mathbf{x}(t))  =  \frac{1}{2} [ \Vert \mathbf{g}(\mathbf{x}(t-1))\Vert^{2}  + \Vert \mathbf{g}(\mathbf{x}(t))\Vert^{2} - \Vert \mathbf{g}(\mathbf{x}(t-1))-\mathbf{g}(\mathbf{x}(t))\Vert^{2} ]. \label{eq:pf-dpp-bound-eq7} 
\end{align} 
Substituting \eqref{eq:pf-dpp-bound-eq7} into \eqref{eq:pf-dpp-bound-eq6} and rearranging terms  yields
\begin{align*}
 &f(\mathbf{x}(t)) + \mathbf{Q}^{T}(t)\mathbf{g}(\mathbf{x}(t)) \nonumber \\
 \leq & f(\mathbf{x}^{\ast}) + \frac{1}{2\gamma} [\Vert \mathbf{x} ^{\ast}- \mathbf{x}(t-1)\Vert^{2} - \Vert \mathbf{x}^{\ast} - \mathbf{x}(t)\Vert^{2}]  + \frac{1}{2}\big[L_{f} + [\mathbf{Q}(t) + \mathbf{g}(\mathbf{x}(t-1))]^{T} \mathbf{L}_{\mathbf{g}} -\frac{1}{\gamma} \big] \Vert \mathbf{x}(t) - \mathbf{x}(t-1)\Vert^{2}   \\ &+ \frac{1}{2}  \Vert \mathbf{g}(\mathbf{x}(t-1))-\mathbf{g}(\mathbf{x}(t))\Vert^{2} - \frac{1}{2}  \Vert \mathbf{g}(\mathbf{x}(t-1))\Vert^{2} - \frac{1}{2}  \Vert \mathbf{g}(\mathbf{x}(t))\Vert^{2}\\
 \overset{(a)}{\leq} &f(\mathbf{x}^{\ast}) + \frac{1}{2\gamma} [\Vert \mathbf{x} ^{\ast}- \mathbf{x}(t-1)\Vert^{2} -\Vert \mathbf{x}^{\ast} - \mathbf{x}(t)\Vert^{2}] + \frac{1}{2} \big[ \beta^{2} + L_{f} + [\mathbf{Q}(t) + \mathbf{g}(\mathbf{x}(t-1))]^{T} \mathbf{L}_{\mathbf{g}} -\frac{1}{\gamma} \big] \Vert \mathbf{x}(t) - \mathbf{x}(t-1)\Vert^{2} \\ &- \frac{1}{2}  \Vert \mathbf{g}(\mathbf{x}(t-1))\Vert^{2} - \frac{1}{2}  \Vert \mathbf{g}(\mathbf{x}(t))\Vert^{2},
 \end{align*}
where (a) follows from the fact that $\Vert \mathbf{g}(\mathbf{x}(t-1)) - \mathbf{g}(\mathbf{x}(t))\Vert \leq \beta \Vert \mathbf{x}(t) - \mathbf{x}(t-1)\Vert$, which further follows from the assumption that $\mathbf{g}(\mathbf{x})$ is Lipschitz continuous with modulus $\beta$.

Summing \eqref{eq:drift} with the above inequality yields
\begin{align*}
&\Delta(t) + f(\mathbf{x}(t)) \\
\leq &f(\mathbf{x}^{\ast}) + \frac{1}{2\gamma} [\Vert \mathbf{x} ^{\ast}- \mathbf{x}(t-1)\Vert^{2} - \Vert \mathbf{x}^{\ast} - \mathbf{x}(t)\Vert^{2}] +\frac{1}{2} [ \Vert \mathbf{g}(\mathbf{x}(t))\Vert^{2} -  \Vert \mathbf{g}(\mathbf{x}(t-1))\Vert^{2} ] \\ &+ \frac{1}{2} \big[ \beta^{2} + L_{f} + [\mathbf{Q}(t) + \mathbf{g}(\mathbf{x}(t-1))]^{T} \mathbf{L}_{\mathbf{g}} -\frac{1}{\gamma} \big] \Vert \mathbf{x}(t) - \mathbf{x}(t-1)\Vert^{2}\\
\overset{(a)}{\leq} &f(\mathbf{x}^{\ast}) + \frac{1}{2\gamma} [\Vert \mathbf{x} ^{\ast}- \mathbf{x}(t-1)\Vert^{2} - \Vert \mathbf{x}^{\ast} - \mathbf{x}(t)\Vert^{2}] +\frac{1}{2} [ \Vert \mathbf{g}(\mathbf{x}(t))\Vert^{2} -  \Vert \mathbf{g}(\mathbf{x}(t-1))\Vert^{2} ] \\ &+ \frac{1}{2} \big[ \beta^{2} + L_{f} + \Vert \mathbf{Q}(t) \Vert \Vert \mathbf{L}_{\mathbf{g}}\Vert + C \Vert \mathbf{L}_{\mathbf{g}} \Vert-\frac{1}{\gamma} \big] \Vert \mathbf{x}(t) - \mathbf{x}(t-1)\Vert^{2},
\end{align*}
where (a) follows from the fact that $ [\mathbf{Q}(t) + \mathbf{g}(\mathbf{x}(t-1))]^{T} \mathbf{L}_{\mathbf{g}}  \leq \Vert \mathbf{Q}(t) + \mathbf{g}(\mathbf{x}(t-1)) \Vert \Vert \mathbf{L}_\mathbf{g}\Vert \leq [ \Vert \mathbf{Q}(t)\Vert + \Vert \mathbf{g}(\mathbf{x}(t))\Vert] \Vert \mathbf{L}_\mathbf{g}\Vert \leq \Vert \mathbf{Q}(t)\Vert \Vert \mathbf{L}_\mathbf{g}\Vert + C \Vert \mathbf{L}_\mathbf{g}\Vert$, where the first step follows from Caucy-Schwartz inequality, the second follows from the triangular inequality of Euclidean norm $\Vert \cdot\Vert$ and the third  follows from the fact that $\Vert \mathbf{g}(\mathbf{x})\Vert \leq C$ for all $\mathbf{x}\in \mathcal{X}$, i.e., Assumption \ref{as:basic}.
\end{IEEEproof}

\begin{Lem}\label{lm:conditional-dpp-bound}
Let $\mathbf{x}^{\ast}$ be an optimal solution of problem \eqref{eq:program-objective}-\eqref{eq:program-set-constraint} and $\boldsymbol{\lambda}^\ast$ be a Lagrange multiplier vector satisfying Assumption \ref{as:strong-duality}. Define 
\begin{align}
D = \beta^{2} + L_{f} + 2 \Vert \boldsymbol{\lambda}^{\ast}\Vert \Vert \mathbf{L}_{\mathbf{g}}\Vert + 2C \Vert \mathbf{L}_{\mathbf{g}}\Vert, \label{eq:D}
\end{align}
where $\beta$, $L_{f}$, $\mathbf{L}_{\mathbf{g}}$ and $C$ are defined in Assumption \ref{as:basic}.  If $\gamma >0$ in Algorithm \ref{alg:new-alg} satisfies
\begin{align}
D+ \Vert \mathbf{L}_{\mathbf{g}}\Vert \frac{R}{\sqrt{\gamma}} - \frac{1}{\gamma} \leq 0, \label{eq:kappa-selection-condition}
\end{align}
where $R$ is defined in Assumption \ref{as:basic}, e.g., 
\begin{align}
0 < \gamma \leq \frac{1}{(\Vert \mathbf{L}_{\mathbf{g}}\Vert R + \sqrt{D})^{2}}, \label{eq:kappa}
\end{align}
then we have
\begin{enumerate}
\item At each iteration $t\in\{0,1,2,\ldots\}$, $\Vert \mathbf{Q}(t)\Vert \leq 2 \Vert \boldsymbol{\lambda}^{\ast}\Vert + \frac{R}{\sqrt{\gamma}} + C$.
\item At each iteration $t\in\{0,1,2,\ldots\}$, $\Delta(t) + f(\mathbf{x}(t)) \leq f(\mathbf{x}^{\ast}) + \frac{1}{2\gamma} [\Vert \mathbf{x} ^{\ast}- \mathbf{x}(t-1)\Vert^{2} - \Vert \mathbf{x}^{\ast} - \mathbf{x}(t)\Vert^{2}] +\frac{1}{2} [ \Vert \mathbf{g}(\mathbf{x}(t))\Vert^{2} -  \Vert \mathbf{g}(\mathbf{x}(t-1))\Vert^{2} ]$.
\end{enumerate}
\end{Lem}
\begin{proof}
Before the main proof, we verify that $\gamma$ given by \eqref{eq:kappa} satisfies \eqref{eq:kappa-selection-condition}. Note that we need to choose $\gamma>0$ such that
\begin{align*}
&D +  \Vert \mathbf{L}_{\mathbf{g}}\Vert \frac{R}{\sqrt{\gamma}}  - \frac{1}{\gamma} \leq 0 \\
\Leftrightarrow & D \gamma +   \Vert \mathbf{L}_{\mathbf{g}}\Vert R \sqrt{\gamma} - 1 \leq 0\\
\Leftrightarrow & 0 < \sqrt{\gamma} \leq \frac{- \Vert \mathbf{L}_{\mathbf{g}}\Vert R  + \sqrt{\Vert \mathbf{L}_{\mathbf{g}}\Vert^{2} R^{2}  + 4D}}{2D} = \frac{2}{\Vert \mathbf{L}_{\mathbf{g}}\Vert R  + \sqrt{\Vert \mathbf{L}_{\mathbf{g}}\Vert^{2} R^{2}  + 4D}}. \\
\end{align*}
Note that 
\begin{align*}
&\frac{2}{\Vert \mathbf{L}_{\mathbf{g}}\Vert R  + \sqrt{\Vert \mathbf{L}_{\mathbf{g}}\Vert^{2} R^{2}  + 4D}} \\
\overset{(a)}{\geq}& \frac{2}{2 \Vert \mathbf{L}_{\mathbf{g}}\Vert  R+ 2\sqrt{D}} \\
= &\frac{1}{\Vert \mathbf{L}_{\mathbf{g}}\Vert R+ \sqrt{D}},
\end{align*}
where (a) follows from the fact that $\sqrt{a+b} \leq \sqrt{a} + \sqrt{b}, \forall a,b\geq 0$.
Thus, if we choose $\sqrt{\gamma} \leq \frac{1}{\Vert \mathbf{L}_{\mathbf{g}}\Vert R+ \sqrt{D}}$, i.e., $0 <\gamma  \leq \frac{1}{( \Vert \mathbf{L}_{\mathbf{g}}\Vert R+ \sqrt{D})^{2}}$, then inequality \eqref{eq:kappa-selection-condition} holds.  

Next, we prove this lemma by induction. 
\begin{itemize}
\item Consider $t=0$. $\Vert \mathbf{Q}(0)\Vert \leq 2  \Vert \boldsymbol{\lambda}^{\ast}\Vert +  \frac{R}{\sqrt{\gamma}} + C$ follows from the fact that $ \Vert \mathbf{Q}(0) \Vert \overset{(a)}{\leq} \Vert \mathbf{g}(\mathbf{x}(-1))\Vert \overset{(b)}{\leq} C$, where (a) follows from part 3 in Lemma \ref{lm:virtual-queue} and (b) follows from Assumption \ref{as:basic}. Thus, the first part in this lemma holds at iteration $t=0$. Note that
\begin{align}
&\beta^{2} + L_{f} + \Vert \mathbf{Q}(0) \Vert \Vert \mathbf{L}_{\mathbf{g}}\Vert + C \Vert \mathbf{L}_{\mathbf{g}} \Vert-\frac{1}{\gamma} \nonumber \\
\overset{(a)}{\leq}&\beta^{2} + L_{f} + \big(2  \Vert \boldsymbol{\lambda}^{\ast}\Vert +  \frac{R}{\sqrt{\gamma}} + C\big)\Vert \mathbf{L}_{\mathbf{g}}\Vert + C \Vert \mathbf{L}_{\mathbf{g}} \Vert-\frac{1}{\gamma} \nonumber \\
=& \beta^{2} + L_{f} + 2  \Vert \boldsymbol{\lambda}^{\ast}\Vert \Vert \mathbf{L}_{\mathbf{g}}\Vert + 2C \Vert \mathbf{L}_{\mathbf{g}}\Vert + \Vert \mathbf{L}_{\mathbf{g}}\Vert  \frac{R}{\sqrt{\gamma}} -\frac{1}{\gamma} \nonumber \\
\overset{(b)}{=}& D + \Vert \mathbf{L}_{\mathbf{g}}\Vert \frac{R}{\sqrt{\gamma}} -\frac{1}{\gamma} \nonumber \\
\overset{(c)}{\leq}& 0, \label{eq:pf-conditional-dpp-bound-eq1}
\end{align}
where (a) follows from $\Vert \mathbf{Q}(0)\Vert \leq 2  \Vert \boldsymbol{\lambda}^{\ast}\Vert +  \frac{R}{\sqrt{\gamma}} + C$; (b) follows from the definition of $D$ in \eqref{eq:D}; and (c) follows from \eqref{eq:kappa-selection-condition}, i.e., the selection rule of $\gamma$.  

Applying Lemma \ref{lm:dpp-bound} at iteration $t=0$ yields
\begin{align*}
\Delta(0) + f(\mathbf{x}(0)) \leq &f(\mathbf{x}^{\ast}) + \frac{1}{2\gamma} [\Vert \mathbf{x} ^{\ast}- \mathbf{x}(-1)\Vert^{2} - \Vert \mathbf{x}^{\ast} - \mathbf{x}(0)\Vert^{2}] +\frac{1}{2} [ \Vert \mathbf{g}(\mathbf{x}(0))\Vert^{2} -  \Vert \mathbf{g}(\mathbf{x}(-1))\Vert^{2} ] \\ &+ \frac{1}{2} \big[ \beta^{2} + L_{f} + \Vert \mathbf{Q}(0) \Vert \Vert \mathbf{L}_{\mathbf{g}}\Vert + C \Vert \mathbf{L}_{\mathbf{g}} \Vert-\frac{1}{\gamma} \big] \Vert \mathbf{x}(0) - \mathbf{x}(-1)\Vert^{2} \\
\overset{(a)}{\leq} & f(\mathbf{x}^{\ast}) + \frac{1}{2\gamma} [\Vert \mathbf{x} ^{\ast}- \mathbf{x}(-1)\Vert^{2} - \Vert \mathbf{x}^{\ast} - \mathbf{x}(0)\Vert^{2}] +\frac{1}{2} [ \Vert \mathbf{g}(\mathbf{x}(0))\Vert^{2} -  \Vert \mathbf{g}(\mathbf{x}(-1))\Vert^{2} ],
\end{align*}
where (a) follows from \eqref{eq:pf-conditional-dpp-bound-eq1}. Thus, the second part in this lemma holds at iteration $t=0$.

\item Assume that for all $\tau\in\{0,1,\ldots,t\}$, we have 
\begin{align}
\Delta(\tau) + f(\mathbf{x}(\tau)) \leq f(\mathbf{x}^{\ast}) + \frac{1}{2\gamma} [\Vert \mathbf{x} ^{\ast}- \mathbf{x}(\tau-1)\Vert^{2} - \Vert \mathbf{x}^{\ast} - \mathbf{x}(\tau)\Vert^{2}] +\frac{1}{2} [ \Vert \mathbf{g}(\mathbf{x}(\tau))\Vert^{2} -  \Vert \mathbf{g}(\mathbf{x}(\tau-1))\Vert^{2}]. \label{eq:pf-conditional-dpp-bound-eq2}
\end{align}
Need to show that both parts in this lemma hold at iteration $t+1$.  

Summing \eqref{eq:pf-conditional-dpp-bound-eq2} over $\tau \in\{0,1,\ldots, t\}$ yields 
\begin{align*}
&\sum_{\tau=0}^{t} \Delta(\tau) + \sum_{\tau=0}^{t} f(\mathbf{x}(\tau)) \\
\leq &(t+1) f(\mathbf{x}^{\ast}) +  \frac{1}{2\gamma} \sum_{\tau=0}^{t}[\Vert \mathbf{x}^{\ast} - \mathbf{x}(\tau-1)\Vert^{2}  - \Vert \mathbf{x}^{\ast} - \mathbf{x}(\tau)\Vert^{2} ] + \frac{1}{2} \sum_{\tau=0}^{t}[\Vert \mathbf{g}(\mathbf{x}(\tau))\Vert^{2}-\Vert \mathbf{g}(\mathbf{x}(\tau-1))\Vert^{2}].
\end{align*}
Recalling that $\Delta(\tau) = L(\tau+1) - L(\tau)$ and simplifying the summations yields
\begin{align*}
&L(t+1)  - L(0) + \sum_{\tau=0}^{t} f(\mathbf{x}(\tau)) \\
 \leq &(t+1) f(\mathbf{x}^{\ast}) +  \frac{1}{2\gamma} \Vert \mathbf{x}^{\ast} - \mathbf{x}(-1)\Vert^{2} -\frac{1}{2\gamma} \Vert \mathbf{x}^{\ast} - \mathbf{x}(t)\Vert^{2}+ \frac{1}{2} \Vert \mathbf{g}(\mathbf{x}(t))\Vert^{2}  -\frac{1}{2} \Vert \mathbf{g}(\mathbf{x}(-1))\Vert^{2} \\
\leq &(t+1) f(\mathbf{x}^{\ast}) +  \frac{1}{2\gamma} \Vert \mathbf{x}^{\ast} - \mathbf{x}(-1)\Vert^{2} + \frac{1}{2} \Vert \mathbf{g}(\mathbf{x}(t))\Vert^{2}  -\frac{1}{2} \Vert \mathbf{g}(\mathbf{x}(-1))\Vert^{2}.
\end{align*}
Rearranging terms yields
\begin{align}
\sum_{\tau=0}^{t} f(\mathbf{x}(\tau)) \leq &(t+1) f(\mathbf{x}^{\ast}) +  \frac{1}{2\gamma} \Vert \mathbf{x}^{\ast} - \mathbf{x}(-1)\Vert^{2} + \frac{1}{2} \Vert \mathbf{g}(\mathbf{x}(t))\Vert^{2}  -\frac{1}{2} \Vert \mathbf{g}(\mathbf{x}(-1))\Vert^{2} + L(0) - L(t+1) \nonumber \\ 
\overset{(a)}{=}  &(t+1) f(\mathbf{x}^{\ast}) +  \frac{1}{2\gamma} \Vert \mathbf{x}^{\ast} - \mathbf{x}(-1)\Vert^{2} + \frac{1}{2} \Vert \mathbf{g}(\mathbf{x}(t))\Vert^{2}  -\frac{1}{2} \Vert \mathbf{g}(\mathbf{x}(-1))\Vert^{2} + \frac{1}{2} \Vert \mathbf{Q}(0)\Vert^{2} - \frac{1}{2} \Vert \mathbf{Q}(t+1)\Vert^{2} \nonumber \\
\overset{(b)}{\leq}& (t+1) f(\mathbf{x}^{\ast}) +  \frac{1}{2\gamma} \Vert \mathbf{x}^{\ast} - \mathbf{x}(-1)\Vert^{2} + \frac{C^{2}}{2} - \frac{1}{2} \Vert \mathbf{Q}(t+1)\Vert^{2} \nonumber\\
\overset{(c)}{\leq} &  (t+1) f(\mathbf{x}^{\ast}) +  \frac{R^{2}}{2\gamma} + \frac{C^{2}}{2}  - \frac{1}{2} \Vert \mathbf{Q}(t+1)\Vert^{2},\label{eq:pf-conditional-dpp-bound-eq3}
\end{align}
where  (a) follows from the definition that $L(0) = \frac{1}{2} \Vert \mathbf{Q}(0)\Vert^{2}$ and $L(t+1) = \frac{1}{2}\Vert \mathbf{Q}(t+1)\Vert^{2}$; (b) follows from the facts that $\Vert \mathbf{g}(\mathbf{x}(t))\Vert \leq C $ which is implied by Assumption \ref{as:basic}, and $\Vert \mathbf{Q}(0)\Vert \leq \Vert \mathbf{g}(\mathbf{x}(-1))\Vert$, i.e., part 3 in Lemma \ref{lm:virtual-queue}; and (c) follows from the fact that $\Vert \mathbf{x} - \mathbf{y}\Vert \leq R$ for all $\mathbf{x}, \mathbf{y}\in \mathcal{X}$, i.e., Assumption \ref{as:basic}.

Applying Lemma \ref{lm:obj-diff-bound-from-strong-duality} at iteration $t+1$ yields
\begin{align}
\sum_{\tau=0}^{t} f(\mathbf{x}(\tau)) \geq (t+1) f(\mathbf{x}^\ast) -  \Vert \boldsymbol{\lambda}^\ast\Vert \Vert \mathbf{Q}(t+1)\Vert.  \label{eq:pf-conditional-dpp-bound-eq4}
\end{align}

Combining \eqref{eq:pf-conditional-dpp-bound-eq3} and \eqref{eq:pf-conditional-dpp-bound-eq4}; and cancelling the common term $(t+1)f(\mathbf{x}^{\ast})$ on both sides yields
\begin{align*}
&\frac{1}{2} \Vert \mathbf{Q}(t+1)\Vert^{2} - \Vert \boldsymbol{\lambda}^{\ast}\Vert \Vert \mathbf{Q}(t+1)\Vert - \frac{R^{2}}{2\gamma}  -\frac{C^{2}}{2} \leq 0 \\
\Rightarrow & ( \Vert \mathbf{Q}(t+1)\Vert - \Vert \boldsymbol{\lambda}^{\ast}\Vert )^{2} \leq \Vert \boldsymbol{\lambda}^{\ast}\Vert^{2}  + \frac{R^{2}}{\gamma}  + C^{2}\\
\Rightarrow &\Vert \mathbf{Q}(t+1)\Vert \leq \Vert \boldsymbol{\lambda}^{\ast}\Vert +\sqrt{ \Vert \boldsymbol{\lambda}^{\ast}\Vert^{2}  + \frac{R^{2}}{\gamma}  + C^{2}} \\
\overset{(a)}{\Rightarrow} & \Vert \mathbf{Q}(t+1)\Vert \leq 2 \Vert \boldsymbol{\lambda}^{\ast}\Vert + \frac{R}{\sqrt{\gamma}} + C,
\end{align*}
where (a) follows from the basic inequality $\sqrt{a+b+c} \leq \sqrt{a} + \sqrt{b} + \sqrt{c}$ for any $a,b,c\geq 0$. Thus, the first part in this lemma holds at iteration $t+1$.

Note that
\begin{align}
&\beta^{2} + L_{f} + \Vert \mathbf{Q}(t+1) \Vert \Vert \mathbf{L}_{\mathbf{g}}\Vert + C \Vert \mathbf{L}_{\mathbf{g}} \Vert-\frac{1}{\gamma} \nonumber \\
\overset{(a)}{\leq}&\beta^{2} + L_{f} + (2  \Vert \boldsymbol{\lambda}^{\ast}\Vert +  \frac{R}{\sqrt{\gamma}} + C)\Vert \mathbf{L}_{\mathbf{g}}\Vert + C \Vert \mathbf{L}_{\mathbf{g}} \Vert-\frac{1}{\gamma} \nonumber \\
=& \beta^{2} + L_{f} + 2  \Vert \boldsymbol{\lambda}^{\ast}\Vert \Vert \mathbf{L}_{\mathbf{g}}\Vert + 2C \Vert \mathbf{L}_{\mathbf{g}}\Vert + \Vert \mathbf{L}_{\mathbf{g}}\Vert \frac{R}{\sqrt{\gamma}} -\frac{1}{\gamma} \nonumber \\
\overset{(b)}{=}& D + \Vert \mathbf{L}_{\mathbf{g}}\Vert \frac{R}{\sqrt{\gamma}} -\frac{1}{\gamma} \nonumber \\
\overset{(c)}{\leq}& 0, \label{eq:pf-conditional-dpp-bound-eq5}
\end{align}
where (a) follows from $\Vert \mathbf{Q}(t+1)\Vert \leq 2  \Vert \boldsymbol{\lambda}^{\ast}\Vert +  \frac{R}{\sqrt{\gamma}} + C$; (b) follows from the definition of $D$ in \eqref{eq:D}; and (c) follows from \eqref{eq:kappa-selection-condition}, i.e., the selection rule of $\gamma$.

Applying Lemma \ref{lm:dpp-bound} at iteration $t+1$ yields
\begin{align*}
\Delta(t+1) + f(\mathbf{x}(t+1)) \leq &f(\mathbf{x}^{\ast}) + \frac{1}{2\gamma} [\Vert \mathbf{x} ^{\ast}- \mathbf{x}(t)\Vert^{2} - \Vert \mathbf{x}^{\ast} - \mathbf{x}(t+1)\Vert^{2}] +\frac{1}{2} [ \Vert \mathbf{g}(\mathbf{x}(t+1))\Vert^{2} -  \Vert \mathbf{g}(\mathbf{x}(t))\Vert^{2}] \\ &+ \frac{1}{2} \big[ \beta^{2} + L_{f} + \Vert \mathbf{Q}(t+1) \Vert \Vert \mathbf{L}_{\mathbf{g}}\Vert + C \Vert \mathbf{L}_{\mathbf{g}} \Vert-\frac{1}{\gamma} \big] \Vert \mathbf{x}(t+1) - \mathbf{x}(t)\Vert^{2} \\
\overset{(a)}{\leq} & f(\mathbf{x}^{\ast}) + \frac{1}{2\gamma} [\Vert \mathbf{x} ^{\ast}- \mathbf{x}(t)\Vert^{2} - \Vert \mathbf{x}^{\ast} - \mathbf{x}(t+1)\Vert^{2}] +\frac{1}{2} [ \Vert \mathbf{g}(\mathbf{x}(t+1))\Vert^{2} -  \Vert \mathbf{g}(\mathbf{x}(t))\Vert^{2} ],
\end{align*}
where (a) follows from \eqref{eq:pf-conditional-dpp-bound-eq5}. Thus, the second part in this lemma holds at iteration $t+1$.

\end{itemize}
Thus, both parts in this lemma follow by induction.
\end{proof}

\begin{Rem}\label{rem:simple-gamma}
Recall that if each $g_k(\mathbf{x})$ is a linear function, then $L_{g_k} = 0$ for all $k\in \{1,2,\ldots,m\}$. In this case, equation \eqref{eq:kappa} reduces to 
\begin{align}
0< \gamma \leq \frac{1}{\beta^2 + L_f}. \label{eq:simple-gamma}
\end{align}
\end{Rem}

\subsection{Objective Value Violations}

\begin{Thm}[Objective Value Violations]
Let $\mathbf{x}^{\ast}$ be an optimal solution of problem \eqref{eq:program-objective}-\eqref{eq:program-set-constraint}. If we choose $\gamma$ according to \eqref{eq:kappa} in Algorithm \ref{alg:new-alg}, then for all $t\geq 1$, we have
\begin{align*}
f(\overline{\mathbf{x}}(t))  \leq  f(\mathbf{x}^{\ast}) + \frac{1}{t}\frac{R^{2}}{2\gamma }.
\end{align*}
where $R$ is defined in Assumption \ref{as:basic}.
\end{Thm}
\begin{IEEEproof}
Fix $t\geq 1$. By part 2 in Lemma \ref{lm:conditional-dpp-bound}, we have $\Delta(\tau) + f(\mathbf{x}(\tau)) \leq f(\mathbf{x}^{\ast}) + \frac{1}{2\gamma} [\Vert \mathbf{x}^{\ast} - \mathbf{x}(\tau-1)\Vert^{2} - \Vert \mathbf{x}^{\ast} - \mathbf{x}(\tau) \Vert] + \frac{1}{2} [ \Vert \mathbf{g}(\mathbf{x}(\tau))\Vert^{2} -  \Vert \mathbf{g}(\mathbf{x}(\tau-1))\Vert^{2} ]$ for all $\tau\in\{0,1,2,\ldots \}$.

Summing over $\tau \in\{0,1,\ldots, t-1\}$ yields 
\begin{align*}
&\sum_{\tau=0}^{t-1} \Delta(\tau) + \sum_{\tau=0}^{t-1} f(\mathbf{x}(\tau)) \\
\leq &t f(\mathbf{x}^{\ast}) +  \frac{1}{2\gamma} \sum_{\tau=0}^{t-1}[\Vert \mathbf{x}^{\ast} - \mathbf{x}(\tau-1)\Vert^{2}  - \Vert \mathbf{x}^{\ast} - \mathbf{x}(\tau)\Vert^{2} ] + \frac{1}{2} \sum_{\tau=0}^{t-1}[\Vert \mathbf{g}(\mathbf{x}(\tau))\Vert^{2}-\Vert \mathbf{g}(\mathbf{x}(\tau-1))\Vert^{2}].
\end{align*}
Recalling that $\Delta(\tau) = L(\tau+1) - L(\tau)$ and simplifying the summations yields
\begin{align*}
&L(t)  - L(0) + \sum_{\tau=0}^{t-1} f(\mathbf{x}(\tau)) \\
 \leq &t f(\mathbf{x}^{\ast}) +  \frac{1}{2\gamma} \Vert \mathbf{x}^{\ast} - \mathbf{x}(-1)\Vert^{2} -\frac{1}{2\gamma} \Vert \mathbf{x}^{\ast} - \mathbf{x}(t-1)\Vert^{2}+ \frac{1}{2} \Vert \mathbf{g}(\mathbf{x}(t-1))\Vert^{2}  -\frac{1}{2} \Vert \mathbf{g}(\mathbf{x}(-1))\Vert^{2} \\
\leq &t f(\mathbf{x}^{\ast}) +  \frac{1}{2\gamma} \Vert \mathbf{x}^{\ast} - \mathbf{x}(-1)\Vert^{2} + \frac{1}{2} \Vert \mathbf{g}(\mathbf{x}(t-1))\Vert^{2}  -\frac{1}{2} \Vert \mathbf{g}(\mathbf{x}(-1))\Vert^{2}. 
\end{align*}
Rearranging terms yields
\begin{align}
\sum_{\tau=0}^{t-1} f(\mathbf{x}(\tau)) \leq &t f(\mathbf{x}^{\ast}) +  \frac{1}{2\gamma} \Vert \mathbf{x}^{\ast} - \mathbf{x}(-1)\Vert^{2} + \frac{1}{2} \Vert \mathbf{g}(\mathbf{x}(t-1))\Vert^{2}  -\frac{1}{2} \Vert \mathbf{g}(\mathbf{x}(-1))\Vert^{2} + L(0) - L(t) \nonumber \\ 
\overset{(a)}{=}  &t f(\mathbf{x}^{\ast}) +  \frac{1}{2\gamma} \Vert \mathbf{x}^{\ast} - \mathbf{x}(-1)\Vert^{2} + \frac{1}{2} \Vert \mathbf{g}(\mathbf{x}(t-1))\Vert^{2}  -\frac{1}{2} \Vert \mathbf{g}(\mathbf{x}(-1))\Vert^{2} + \frac{1}{2} \Vert \mathbf{Q}(0)\Vert^{2} - \frac{1}{2} \Vert \mathbf{Q}(t)\Vert^{2} \nonumber \\
\overset{(b)}{\leq}& t f(\mathbf{x}^{\ast}) +  \frac{1}{2\gamma} \Vert \mathbf{x}^{\ast} - \mathbf{x}(-1)\Vert^{2} +  \frac{1}{2} \Vert \mathbf{g}(\mathbf{x}(t-1))\Vert^{2} - \frac{1}{2} \Vert \mathbf{Q}(t)\Vert^{2} \nonumber\\
\overset{(c)}{\leq} & t f(\mathbf{x}^{\ast}) +  \frac{1}{2\gamma} \Vert \mathbf{x}^{\ast} - \mathbf{x}(-1)\Vert^{2} \nonumber \\
\overset{(d)}{\leq} & t f(\mathbf{x}^{\ast}) +  \frac{R^{2}}{2\gamma},  
\end{align}
where  (a) follows from the definition that $L(0) = \frac{1}{2} \Vert \mathbf{Q}(0)\Vert^{2}$ and $L(t) = \frac{1}{2}\Vert \mathbf{Q}(t)\Vert^{2}$; (b) follows from the fact that $\Vert \mathbf{Q}(0)\Vert \leq \Vert \mathbf{g}(\mathbf{x}(-1))\Vert$, i.e., part 3 in Lemma \ref{lm:virtual-queue}; (c) follows from the fact that $\Vert \mathbf{Q}(t)\Vert^{2} \geq \Vert \mathbf{g}(\mathbf{x}(t-1))\Vert^{2}$ when $t\geq 1$, i.e., part 3 in Lemma \ref{lm:virtual-queue}; and (d) follows from the fact that $\Vert \mathbf{x} - \mathbf{y}\Vert\leq R$ for all $\mathbf{x}, \mathbf{y}\in \mathcal{X}$, i.e., Assumption \ref{as:basic}.

Dividing both sides by factor $t$ yields
\begin{align*}
\frac{1}{t} \sum_{\tau=0}^{t-1} f(\mathbf{x}(\tau)) \leq f(\mathbf{x}^{\ast}) +  \frac{1}{t}\frac{R^{2}}{2\gamma }. \end{align*}

Finally, since $\overline{\mathbf{x}}(t) = \frac{1}{t} \sum_{\tau=0}^{t-1} \mathbf{x}(\tau)$ and $f(\mathbf{x})$ is convex, By Jensen's inequality it follows that 
\begin{align*}
f(\overline{\mathbf{x}}(t)) \leq \frac{1}{t} \sum_{\tau=0}^{t-1} f(\mathbf{x}(\tau)).
\end{align*}
\end{IEEEproof}

The above theorem shows that the error gap between $f(\overline{\mathbf{x}}(t))$ and the optimal value $f(\mathbf{x}^{\ast})$ delays like $O(1/t)$.

\subsection{Constraint Violations}

\begin{Thm}[Constraint Violations]
Let $\mathbf{x}^{\ast}$ be an optimal solution of problem \eqref{eq:program-objective}-\eqref{eq:program-set-constraint} and $\boldsymbol{\lambda}^\ast$ be a Lagrange multiplier vector satisfying Assumption \ref{as:strong-duality}.  If we choose $\gamma$ according to \eqref{eq:kappa} in Algorithm \ref{alg:new-alg}, then for all $t\geq 1$, the constraint functions satisfy
\begin{align*}
g_{k}(\overline{\mathbf{x}}(t)) \leq  \frac{1}{t} \big(2 \Vert \boldsymbol{\lambda}^\ast \Vert + \frac{R}{\sqrt{\gamma}} + C \big), \forall k\in\{1,2,\ldots, m\}.
\end{align*}
where $R$ and $C$ are defined in Assumption \ref{as:basic}.
\end{Thm}
\begin{IEEEproof}
Fix $t\geq 1$ and $k\in\{1,2,\ldots,m\}$. Recall that $\overline{\mathbf{x}}(t) = \frac{1}{t}\sum_{\tau=0}^{t-1} \mathbf{x}(\tau)$. Thus, 
\begin{align*}
g_k (\overline{\mathbf{x}}(t)) &\overset{(a)}{\leq} \frac{1}{t} \sum_{\tau=0}^{t-1} g_k(\mathbf{x}(\tau)) \\
 &\overset{(b)}{\leq} \frac{Q_k(t)}{t} \\
 &\leq \frac{\Vert \mathbf{Q}(t)\Vert}{t}\\
 &\overset{(c)}{\leq} \frac{1}{t} \big( 2 \Vert \boldsymbol{\lambda}^\ast \Vert + \frac{R}{\sqrt{\gamma}} + C \big),
 \end{align*}
where (a) follows from the convexity of $g_k(\mathbf{x}), k\in\{1,2,\ldots,m\}$ and Jensen's inequality; (b) follows from Lemma \ref{lm:queue-constraint-inequality}; and (c) follows from part 1 in Lemma \ref{lm:conditional-dpp-bound}.
\end{IEEEproof}

\subsection{Convergence Rate of Algorithm \ref{alg:new-alg}}

The next theorem summarizes the last two subsections.
\begin{Thm}\label{thm:overall-convergence}
Let $\mathbf{x}^{\ast}$ be an optimal solution of problem \eqref{eq:program-objective}-\eqref{eq:program-set-constraint} and $\boldsymbol{\lambda}^\ast$ be a Lagrange multiplier vector satisfying Assumption \ref{as:strong-duality}. If we choose $\gamma$ according to \eqref{eq:kappa} in Algorithm \ref{alg:new-alg}, then for all $t\geq 1$, we have
\begin{align*}
f(\overline{\mathbf{x}}(t))  \leq &  f(\mathbf{x}^{\ast}) + \frac{1}{t}\frac{R^{2}}{2\gamma}, \\
g_{k}(\overline{\mathbf{x}}(t)) \leq&  \frac{1}{t} \big(2 \Vert \boldsymbol{\lambda}^\ast \Vert + \frac{R}{\sqrt{\gamma}} + C \big), \forall k\in\{1,2,\ldots, m\},
\end{align*}
where $C$ and $R$ are defined in Assumption \ref{as:basic}. In summary, Algorithm \ref{alg:new-alg} ensures error decays like $O(1/t)$ and provides an $\epsilon$-optimal solution with convergence time $O(1/\epsilon)$.
\end{Thm}

\subsection{Practical Implementations}
By Theorem \ref{thm:overall-convergence}, it suffices to choose $\gamma$ according to \eqref{eq:kappa} to guarantee the $O(1/t)$ convergence rate of Algorithm \ref{alg:new-alg}. By Remark \ref{rem:simple-gamma}, if all constraint functions are linear, then \eqref{eq:kappa} reduces to \eqref{eq:simple-gamma} and is independent of $\Vert \boldsymbol{\lambda}^\ast\Vert$. For general constraint functions, we need to know the value of $\Vert \boldsymbol{\lambda}^{\ast}\Vert$, which is typically unknown, to select $\gamma$ according to \eqref{eq:kappa}.  However, it is easy to observe that an upper bound of $\Vert \boldsymbol{\lambda}^{\ast}\Vert$ is sufficient for us to choose $\gamma$ satisfying \eqref{eq:kappa}. To obtain an upper bound of $\Vert \boldsymbol{\lambda}^{\ast}\Vert$, the next lemma is useful if problem \eqref{eq:program-objective}-\eqref{eq:program-set-constraint} has an interior feasible point, i.e., the Slater's condition is satisfied.

The next lemma from \cite{Nedic09} provides an upper bound of the Lagrangian multipliers associated with the zero duality gap under the Slater's condition.
\begin{Lem}[Lemma 1 in \cite{Nedic09}] \label{lm:Nedic-multiplier-upper-bound}
Consider the convex program given by 
\begin{align*}
\min_{\mathbf{x}} \quad &  f(\mathbf{x})\\
\text{s.t.} \quad  & g_{k}(\mathbf{x})\leq 0, k\in\{1,2,\ldots,m\}\\
			 & \mathbf{x}\in \mathcal{X} \subseteq{R}^{n}
\end{align*}
and define the Lagrangian dual function as $q(\boldsymbol{\lambda}) = \inf_{\mathbf{x}\in \mathcal{X}}\{f(\mathbf{x}) + \boldsymbol{\lambda}^{T} \mathbf{g}(\mathbf{x})\}$. If the Slater's condition holds, i.e., there exists $
\hat{\mathbf{x}}\in X$ such that $g_{j}(\mathbf{x}) < 0, \forall j\in\{1,2,\ldots, m\}$, then the level sets $\mathcal{V}_{\hat{\boldsymbol{\lambda}}} = \{\boldsymbol{\lambda}\geq \mathbf{0} :  q(\boldsymbol{\lambda}) \geq q(\hat{\boldsymbol{\lambda}})\}$ is bounded for any $\hat{\boldsymbol{\lambda}}$. In particular, we have
\begin{align*}
\max_{\boldsymbol{\lambda} \in \mathcal{V}_{\hat{\boldsymbol{\lambda}}}} \Vert \boldsymbol{\lambda}\Vert \leq \frac{1}{\min_{1\leq j\leq m}\{-g_{j}(\hat{\mathbf{x}})\}} (f(\hat{\mathbf{x}}) - q(\hat{\boldsymbol{\lambda}})).
\end{align*}
\end{Lem}

By Lemma \ref{lm:Nedic-multiplier-upper-bound}, if convex program  \eqref{eq:program-objective}-\eqref{eq:program-set-constraint} has a feasible point $\hat{\mathbf{x}}\in \mathcal{X}$ such that $g_{k}(\hat{\mathbf{x}}) < 0, \forall k\in\{1,2,\ldots, m\}$, then we can take an arbitrary $\hat{\boldsymbol{\lambda}} \geq \mathbf{0}$ to obtain the value $q(\hat{\boldsymbol{\lambda}}) = \inf_{\mathbf{x}\in \mathcal{X}}\{f(\mathbf{x}) + \hat{\boldsymbol{\lambda}}^{T} \mathbf{g}(\mathbf{x})\}$ and conclude that $\Vert 
\boldsymbol{\lambda}^{\ast}\Vert \leq \frac{1}{\min_{1\leq j\leq m}\{-g_{j}(\hat{\mathbf{x}})\}} (f(\hat{\mathbf{x}}) - q(\hat{\boldsymbol{\lambda}}))$.  Since $f(\mathbf{x})$ is continuous and $\mathcal{X}$ is a compact set, there exists a constant $F>0$ such that $|f(\mathbf{x})|\leq F$ for all $\mathbf{x}\in \mathcal{X}$. Thus, we can take $\hat{\boldsymbol{\lambda}}= \mathbf{0}$ such that $q(\hat{\boldsymbol{\lambda}}) = \min_{\mathbf{x}\in \mathcal{X}}\{f(\mathbf{x})\} \geq -F$. It follows from Lemma \ref{lm:Nedic-multiplier-upper-bound} that $\Vert \boldsymbol{\lambda}^\ast\Vert \leq \frac{1}{\min_{1\leq j\leq m}\{-g_{j}(\hat{\mathbf{x}})\}} (f(\hat{\mathbf{x}}) - q(\hat{\boldsymbol{\lambda}})) \leq \frac{2F}{\min_{1\leq j\leq m}\{-g_{j}(\hat{\mathbf{x}})\}}$.

\section{Numerical Results}
\subsection{Linear Programs with Box Constraints}
Consider the following linear program
\begin{align*}
\min~~ &  \mathbf{c}^T \mathbf{x}\\
\text{s.t.} \quad  & \mathbf{A}\mathbf{x} \leq \mathbf{b}\\
			 & \mathbf{x} \in \mathcal{X} =\{\mathbf{x}: \mathbf{x}^{\min} \leq \mathbf{x} \leq \mathbf{x}^{\max}\} \subseteq \mathbf{R}^n
\end{align*}
In general, we are interested in the linear programs where the optimum is attained by a finite $\mathbf{x}^\ast$.  The box constraints are in general needed to enforce the compactness of the constraint set and appear in practical applications since decision variable $\mathbf{x}$ in engineering problems should be taken from a bounded set.  Such a linear program can arise in applications like multi-commodity network flow problems, portfolio optimization and etc.

Note that the object function $f(\mathbf{x}) = \mathbf{c}^T \mathbf{x}$ is smooth with modulus $L_f = 0$; and the constraint function $\mathbf{g}(\mathbf{x}) = \mathbf{A} \mathbf{x} - \mathbf{b}$ is smooth with modulus $\mathbf{L}_\mathbf{g} = \mathbf{0}$ and is Lipschitz with modulus $\beta = \Vert \mathbf{A}\Vert$, where $\Vert \mathbf{A} \Vert$ represents the $L_2$ norm of matrix $\mathbf{A}$.

Since both the object and constraint functions in linear programs are separable,  Algorithm \ref{alg:general-alg} can also be applied to solve it and at each iteration the update of $\mathbf{x}(t)$ only requires to solve $n$  one-dimension set constrained problems, which  have closed-form solutions. Thus,  Algorithm \ref{alg:new-alg} does not have obvious complexity advantages over Algorithm \ref{alg:general-alg}. The numerical experiment performed in this subsection is only to verify the $O(1/t)$ convergence rate of Algorithm \ref{alg:new-alg}. 

Now consider an instance of the linear program with $\mathbf{c} = [-1,-4,-3,-2]^T$, $\mathbf{A} = \left[ \begin{array}{cccc} 6 &1 & 5 &1\\ 0 &3 & 6 &6 \\ 5& 6 & 4 & 6\end{array}\right]$, $\mathbf{b} = [6,4,10]^T$, $\mathbf{x}^{\min} = [0,0,0,0]^T$ and $\mathbf{x}^{\max} = [10,10,10,10]^T$. The optimal solution to this linear program is $\mathbf{x}^\ast = [0.4, \frac{4}{3}, 0, 0]^T$ and the optima value is $f^\ast = -5.73333$.  To verify the convergence, Figure \ref{fig:lp_convergence} shows the value of object and constraint functions yielded by Algorithm \ref{alg:new-alg} with $\mathbf{x}(-1) = \mathbf{x}^{\max} = [10,10,10,10]^T$ and $\gamma = \Vert \mathbf{\mathbf{A}}\Vert_F^2= 1/257$ since $\Vert \mathbf{A}\Vert_F \geq \Vert \mathbf{A}\Vert$, where $\Vert \mathbf{A}\Vert_F$  is the Frobenius norm of matrix $\mathbf{A}$. Note that by Theorem \ref{thm:overall-convergence} and equation \eqref{eq:kappa}, to guarantee the $O(1/t)$ convergence rate of Algorithm \ref{alg:new-alg}, it suffices to choose $\gamma \leq \frac{1}{\beta^2} = \frac{1}{\Vert \mathbf{A}\Vert^2}$. To verify the $O(1/t)$ convergence rate, Figure \ref{fig:lp_convergence_rate} plots $f^\ast - f(\overline{\mathbf{x}}(t))$, linear constraint function values and curve $1/t$ with both x-axis and y-axis in $\log_{10}$ scales.  It can observed that the curve of $f^\ast - f(\overline{\mathbf{x}}(t))$ is parallel to the the curve of $1/t$ for large $t$. Note that the linear constraints are satisfied very early (i.e., negative starting from iteration $t=7$) although two among the three linear constraints are actually tight at the optimal solution; and hence are not drawn in $\log_{10}$ scale after  iteration $t=7$. Figure \ref{fig:lp_convergence_rate} verifies that the error of Algorithm \ref{alg:new-alg} decays like $O(1/t)$ and suggests that it is actually $\Theta(1/t)$ for this linear program.  

\begin{figure}[htbp]
\centering
   \includegraphics[width=0.8\textwidth,height=0.8\textheight,keepaspectratio=true]{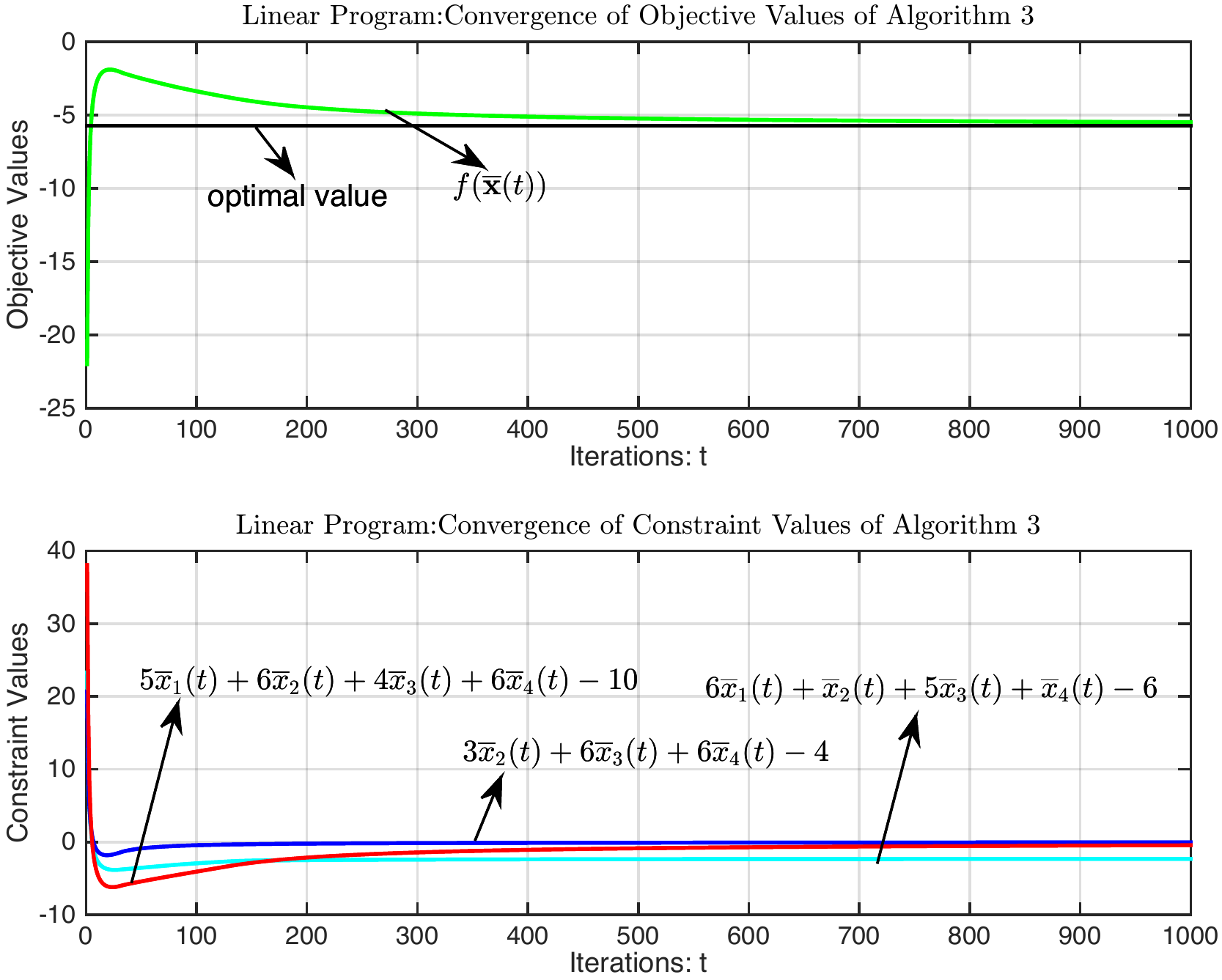} 
   \caption{The convergence of Algorithm \ref{alg:new-alg} for a linear program.}
   \label{fig:lp_convergence}
\end{figure}

\begin{figure}[htbp]
\centering
   \includegraphics[width=0.8\textwidth,height=0.8\textheight,keepaspectratio=true]{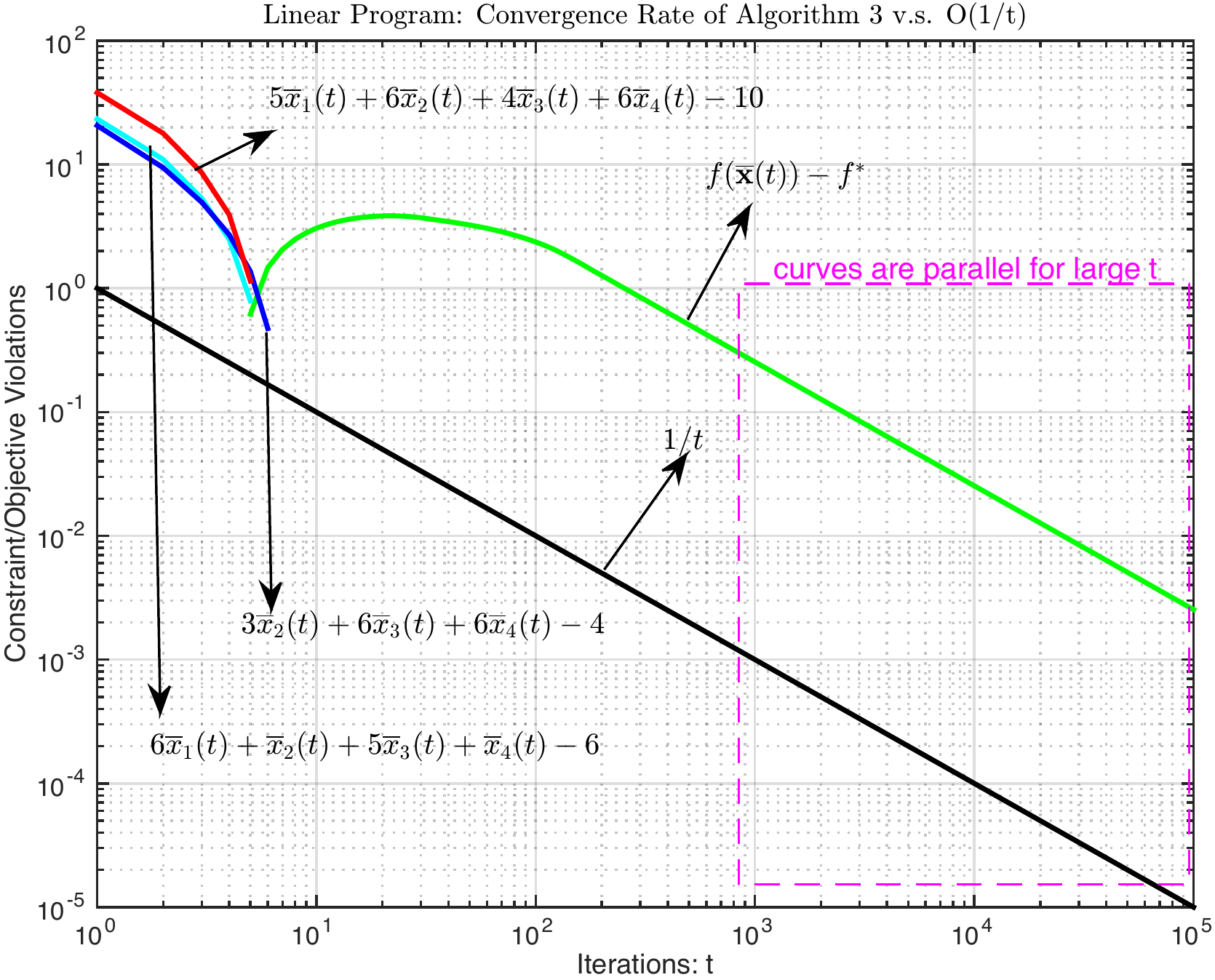} 
   \caption{The convergence rate of Algorithm \ref{alg:new-alg} for a linear program.}
   \label{fig:lp_convergence_rate}
\end{figure}

In fact, it can be verified that the dynamic of Algorithm \ref{alg:new-alg} is identical to that of Algorithm \ref{alg:general-alg} with $\alpha = 1/(2\gamma)$ if they take the same initial $\mathbf{x}(-1)$. This is because the gradient based update of $\mathbf{x}(t)$ in Algorithm \ref{alg:new-alg} can be interpreted  as the minimization of sum of the first-order expansion of $\phi(\mathbf{x}) = f(\mathbf{x}) + [\mathbf{Q}(t) + \mathbf{g}(\mathbf{x}(t-1))]^T \mathbf{g}(\mathbf{x})$ around point $\mathbf{x} = \mathbf{x}(t-1)$ and $\frac{1}{2\gamma} \Vert \mathbf{x} - \mathbf{x}(t-1)\Vert^2$ (as observed in equation \eqref{eq:pf-conditional-dpp-bound-eq1}). In the case when both $f(\mathbf{x})$ and $\mathbf{g}(\mathbf{x})$ are linear, the first order expansion of $\phi(\mathbf{x})$ is identical to itself while the update of  $\mathbf{x}(t)$ in Algorithm \ref{alg:general-alg} is to minimize $\phi(\mathbf{x}) + \alpha \Vert \mathbf{x} - \mathbf{x}(t-1)\Vert$. Thus, if $\alpha = 1/(2\gamma)$,  then Algorithm \ref{alg:new-alg} identical to Algorithm \ref{alg:general-alg}.

\subsection{Quadratic Programs with Box Constraints}
Consider the following quadratic program
\begin{align*}
\min~~ &  \mathbf{x}^T\mathbf{P}\mathbf{x} + \mathbf{c}^T \mathbf{x}\\
\text{s.t.} \quad  & \mathbf{A}\mathbf{x} \leq \mathbf{b}\\
			 & \mathbf{x}^T \mathbf{Q} \mathbf{x} + \mathbf{d}^T \mathbf{x}\leq e\\
			 & \mathbf{x} \in \mathcal{X} =\{\mathbf{x}: \mathbf{x}^{\min} \leq \mathbf{x} \leq \mathbf{x}^{\max}\}
\end{align*}
where both $\mathbf{P}$ and $\mathbf{Q}$ are symmetric and positive semidefinite to ensure the convexity of the quadratic program.

Note that the object function $f(\mathbf{x}) = \mathbf{x}^T\mathbf{P}\mathbf{x} + \mathbf{c}^T \mathbf{x}$ is smooth with modulus $L_f = \Vert \mathbf{P}\Vert$, where $\Vert \mathbf{P}\Vert$ represents the $L_2$ norm of matrix $\mathbf{P}$; and the constraint function $\mathbf{g}(\mathbf{x}) = [\mathbf{A} \mathbf{x} - \mathbf{b}; \mathbf{x}^T \mathbf{Q} \mathbf{x}-\mathbf{d}^T\mathbf{x} -e]^T$ is smooth with modulus $ \Vert \mathbf{L}_\mathbf{g} \Vert =  \Vert \mathbf{Q}\Vert$ and is Lipschitz continuous with modulus $\beta \leq \Vert \mathbf{A}\Vert + \max_{\mathbf{x}\in \mathcal{X}}[ \Vert 2\mathbf{x}^T \mathbf{Q} + \mathbf{d}^T\Vert_2] \leq \Vert \mathbf{A}\Vert+ 2 \Vert \mathbf{Q}\Vert (\Vert \mathbf{x}^{\min}\Vert + \Vert \mathbf{x}^{\max}\Vert) + \Vert \mathbf{d}\Vert $; constant $C$ satisfying Assumption \ref{as:basic} can be given by $C =  \Vert \mathbf{A}\Vert (\Vert \mathbf{x}^{\min}\Vert + \Vert \mathbf{x}^{\max}\Vert) + \Vert \mathbf{b}\Vert + \Vert \mathbf{Q}\Vert (\Vert \mathbf{x}^{\min}\Vert + \Vert \mathbf{x}^{\max}\Vert)^2 + \Vert \mathbf{d}\Vert (\Vert \mathbf{x}^{\min}\Vert + \Vert \mathbf{x}^{\max}\Vert) + |e|$; and constant $R$ satisfying Assumption \ref{as:basic} can be given by $R = \Vert \mathbf{x}^{\min}\Vert + \Vert \mathbf{x}^{\max}\Vert$.

Note that if $\mathbf{P}$ or $\mathbf{Q}$ are not diagonal, then the object function or constraint functions are not separable and hence at each iteration the update of $\mathbf{x}(t)$ in Algorithm \ref{alg:general-alg} requires to solve an $n$-dimensional set constrained quadratic program, which can have huge complexity when $n$ is large.  In contrast, the update of $\mathbf{x}(t)$ in Algorithm \ref{alg:new-alg} has much smaller complexity. 

Now consider an instance of the quadratic program with $\mathbf{P} = \left[ \begin{array}{cc} 1 & 2 \\ 2 &4 \end{array}\right]$, $\mathbf{c} = [-8,-2]^T$, $\mathbf{A} = \left[ \begin{array}{cc} 3 & 1 \\ 2 & 2 \end{array}\right]$, $\mathbf{b} = [4,1]^T$, $\mathbf{Q} = \left[ \begin{array}{cc} 2 & 1 \\ 1 &3 \end{array}\right]$, $\mathbf{d} = [-1,2]^T$, $e = 5$, $\mathbf{x}^{\min} = [0,0]^T$ and $\mathbf{x}^{\max} = [5,5]^T$. The optimal solution to this quadratic program is $\mathbf{x}^\ast = [0.5,0]^T$ and the optimal value is $f^\ast = -3.75$.

Note that $\hat{\mathbf{x}} = [0,0]^T$ is an interior point with $\min_{1\leq j\leq m}\{-g_j(\hat{\mathbf{x}})\} = 1$ and if $\hat{\boldsymbol{\lambda}} = [0,0,0]^T$, then $q(\hat{\boldsymbol{\lambda}}) = \min_{\mathbf{x}\in \mathcal{X}} \{f(\mathbf{x}) + \hat{\boldsymbol{\lambda}}^T \mathbf{g}(\mathbf{x})\} =  \min_{\mathbf{x}\in \mathcal{X}} \{\mathbf{x}^T \mathbf{P}\mathbf{x} + \mathbf{c}^T \mathbf{x}\} \geq \min_{\mathbf{x}\in \mathcal{X}} \{\mathbf{x}^T \mathbf{P}\mathbf{x}\} \min_{\mathbf{x}\in \mathcal{X}} + \mathbf{c}^T \mathbf{x}\} \geq -50$. Thus, by Lemma \ref{lm:Nedic-multiplier-upper-bound}, we know $\Vert \boldsymbol{\lambda}^\ast\Vert \leq 50$ for any $\boldsymbol{\lambda}^\ast$ attaining strong duality.  It can be checked that $\gamma \leq 0.1395$ satisfies equation \eqref{eq:kappa} and hence ensures the $O(1/t)$ convergence rate of Algorithm \ref{alg:new-alg} by Theorem \ref{thm:overall-convergence}. To verify the convergence, Figure \ref{fig:qp_convergence} shows the value of object and constraint functions yielded by Algorithm \ref{alg:new-alg} with $\mathbf{x}(-1) = \mathbf{x}^{\min} = [0,0]^T$ and $\gamma = 0.1395$. To verify the $O(1/t)$ convergence rate, Figure \ref{fig:qp_convergence_rate} plots $f^\ast - f(\overline{\mathbf{x}}(t))$, $\mathbf{g}(\overline{\mathbf{x}}(t))$ and function $h(t) = 1/t$ with both x-axis and y-axis in $\log_{10}$ scales.  It can observed that the curves of $f^\ast - f(\overline{\mathbf{x}}(t))$ and $2\overline{x}_1(t) + 2\overline{x}_2(t) -1$ are  parallel to the  curve of $1/t$ for large $t$. Note that the other linear constraint $3x_1 + x_2 -4\leq 0$ and quadratic constraint $\mathbf{x}^T\mathbf{Q} \mathbf{x} + \mathbf{d}^T \mathbf{x} -e\leq 0$ are satisfied starting from iteration $t=1$ (i.e., negative starting from iteration $t=1$) as also observed in Figure \ref{fig:qp_convergence} and hence are not drawn in $\log_{10}$ scale. Figure \ref{fig:qp_convergence_rate} verifies that the error of Algorithm \ref{alg:new-alg} decays like $O(1/t)$ and suggests that it is actually $\Theta(1/t)$ for this linear program.  

\begin{figure}[htbp]
\centering
   \includegraphics[width=0.8\textwidth,height=0.8\textheight,keepaspectratio=true]{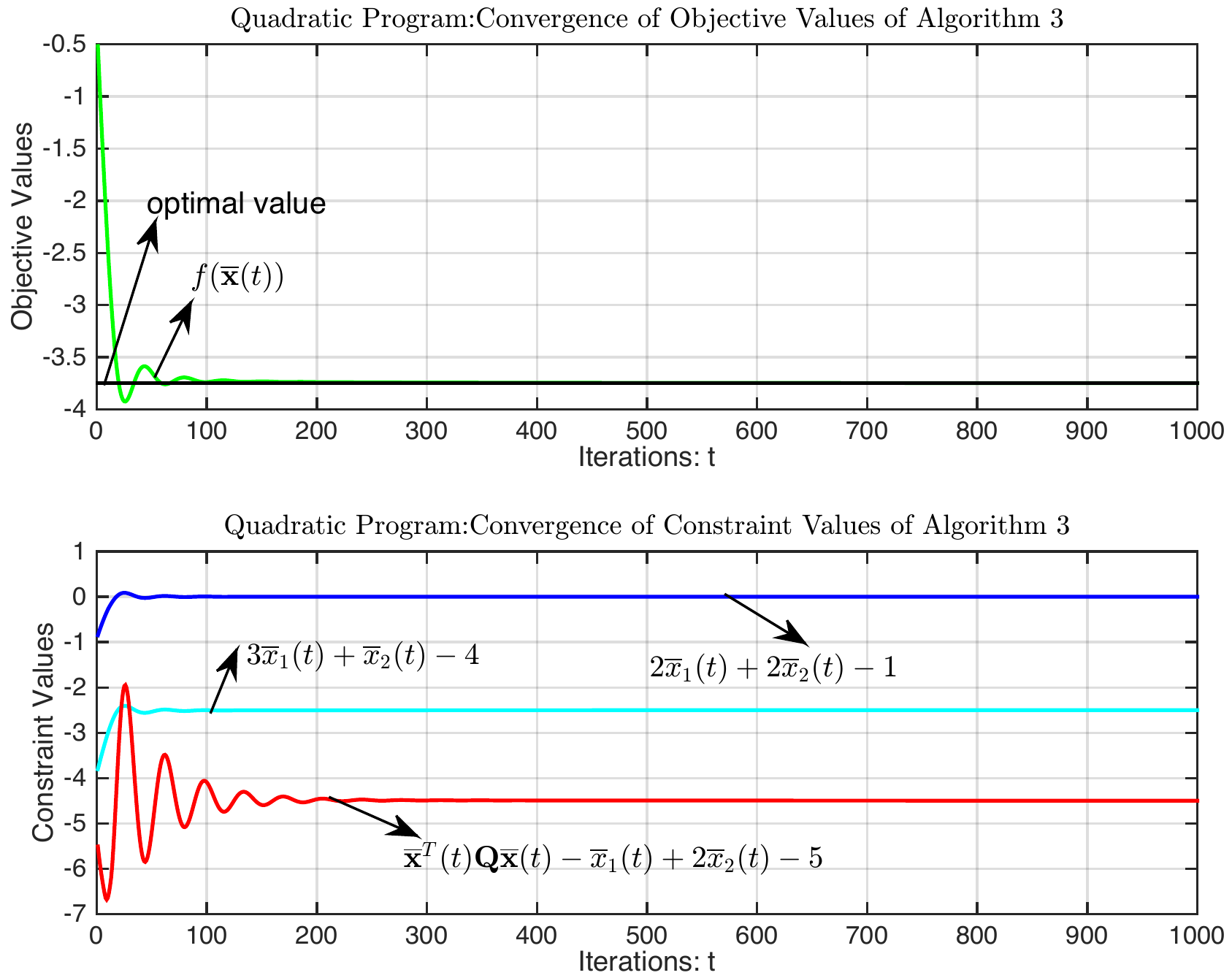} 
   \caption{The convergence of Algorithm \ref{alg:new-alg} for a quadratic program.}
   \label{fig:qp_convergence}
\end{figure}

\begin{figure}[htbp]
\centering
   \includegraphics[width=0.8\textwidth,height=0.8\textheight,keepaspectratio=true]{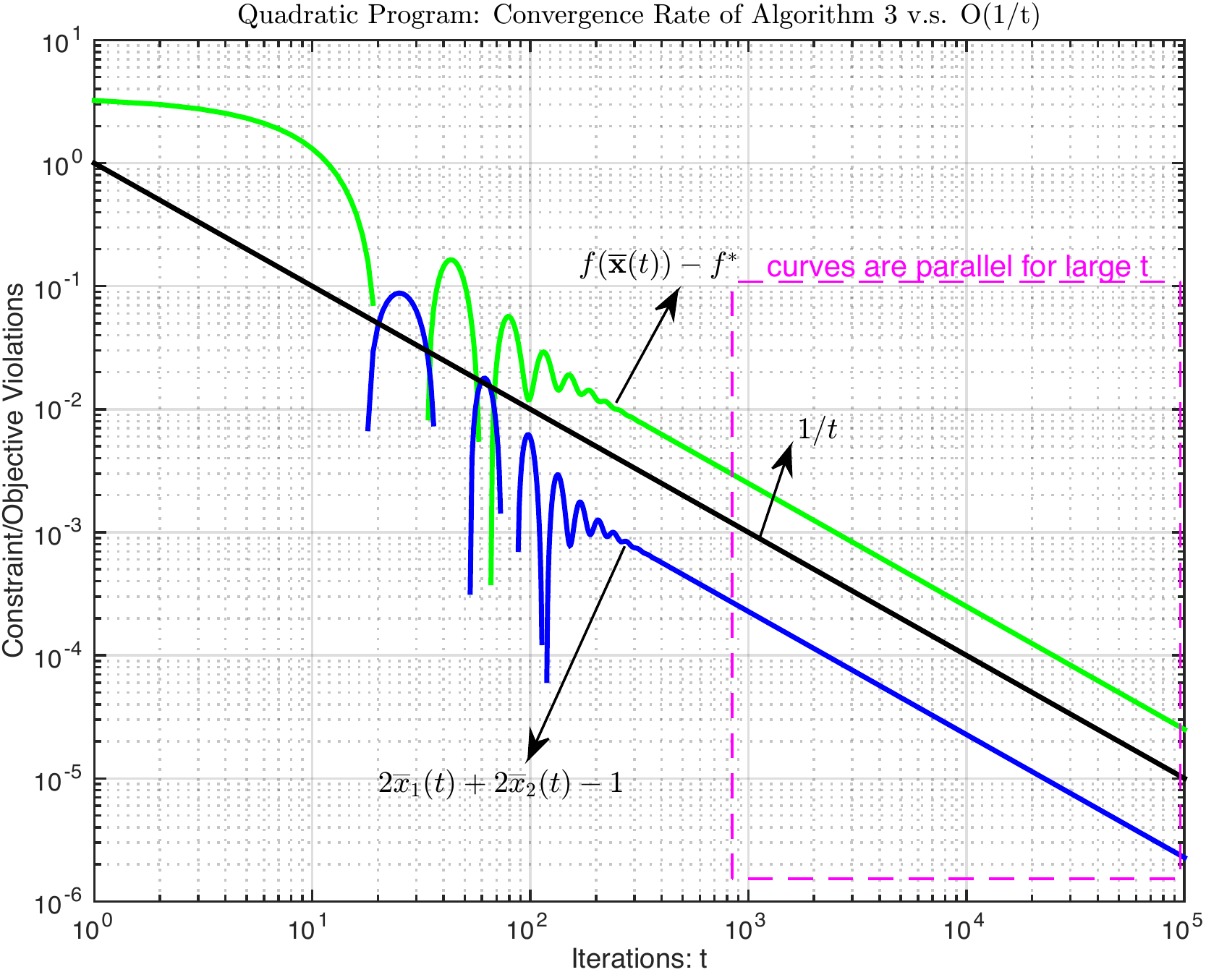} 
   \caption{The convergence rate of Algorithm \ref{alg:new-alg} for a quadratic program.}
   \label{fig:qp_convergence_rate}
\end{figure}

\section{Conclusion}

This paper proposes a new primal-dual type algorithm with the $O(1/t)$ convergence rate for constrained convex programs. At each iteration, the new algorithm updates the primal variable $\mathbf{x}(t)$ following simple gradient updates and hence is suitable to large scale convex programs. The convergence rate of the new algorithm is faster than the $O(1/\sqrt{t})$ convergence rate of the classical primal-dual subgradient algorithm or the dual subgradient algorithm.  The new algorithm has the same convergence rate as a  parallel algorithm recently proposed in \cite{YuNeely15ArXivGeneralConvex} for convex programs with separable object and constraint functions. However, if the object or constraint function is not separable, the algorithm in \cite{YuNeely15ArXivGeneralConvex} is no longer parallel and each iteration requires to solve a set constrained convex program. In contrast, the algorithm proposed in this paper only involves a simple gradient update with low complexity at each iteration.  Thus,  the new algorithm has much smaller per-iteration complexity than the algorithm in \cite{YuNeely15ArXivGeneralConvex}.

\bibliographystyle{IEEEtran}
\bibliography{IEEEfull,mybibfile}

\end{document}